\documentclass[12pt]{article}
\usepackage{amsmath,amsfonts,amssymb}
\usepackage{alltt}

\usepackage{amsmath,amsfonts,ifthen,fullpage,enumerate}  %,amsthm}
\usepackage{mathrsfs}
\usepackage{mathtools}

\usepackage{enumerate}

\usepackage{amsmath,amsfonts,amssymb}

\usepackage{makecell}  %thick hline

\usepackage{graphicx}        % standard LaTeX graphics tool
\usepackage[table]{xcolor}   % For coloured cells.

\def\Co{\mathop{\rm Co}\nolimits}

\def\spam{\mathop{\rm span}\nolimits}

\def\Re{\mathop{\rm Re}\nolimits}
\def\Im{\mathop{\rm Im}\nolimits}

\def\trace{\mathop{\rm trace}\nolimits}

\def\rank{\mathop{\rm rank}\nolimits}

\def\diag{\mathop{\rm diag}\nolimits}

\def\pmat#1{\begin{pmatrix}#1\end{pmatrix}}

\def\question#1{{\bf Question: }#1}
\def\question#1{}

\def\R{\mathbb{R}}

\def\CC{\mathbb{C}}

\def\FF{\mathbb{F}}
\def\HH{\mathbb{H}}

\def\OO{\mathbb{O}}
\def\PP{\mathbb{P}}

\def\Cd{\C^d}
\def\Fd{\FF^d}

\def\Hd{\HH^d}
\def\Hn{\HH^n}
\def\Rd{\R^d}

\def\C{\mathbb{C}}

\def\implies{\Longrightarrow}

\newcommand{\RR}{\mathbb{R}}

\newtheorem{theorem}{Theorem}[section]

\newtheorem{lemma}{Lemma}[section]
\newtheorem{example}{Example}[section]

\newtheorem{proposition}{Proposition}[section]
\newtheorem{definition}{Definition}[section]
\newtheorem{conjecture}{Conjecture}
\newenvironment{proof}{{\noindent \it
Proof.}}{\hfill$\Box$\medskip}
\newif\ifdraft\def\draft{\drafttrue\hoffset=.8truecm\showlabeltrue
\def\comment##1{{\bf comment: ##1}}
\headline={\sevenrm \hfill \ifx\filenamed\undefined\jobname\else\filenamed\fi%
(.tex) (as of \ifx\updated\undefined???\else\updated\fi)
 \TeX'ed at {\hour\time\divide\hour by 60{}%
\minutes\hour\multiply\minutes by 60{}%
\advance\time by -\minutes
\the\hour:\ifnum\time<10{}0\fi\the\time\  on \today\hfill}}
}

\def\inpro#1{\langle#1\rangle}
\def\ip#1{\langle\kern-.28em\langle#1\rangle\kern-.28em\rangle_\nu}

\def\cH{{\cal H}}

\def\cV{{\cal V}}
\def\cW{{\cal W}}

\def\norm#1{\Vert#1\Vert}%\def\norm#1{\|#1\|} does not work with verbatim.tex
\def\openR{{{\rm I}\kern-.16em {\rm R}}}

\def\Fd{\FF^d}
\let\ga\alpha

\let\gb\beta

\let\gD\Delta

\let\gth\theta

\let\gl\lambda

\let\gL\Lambda

\let\gs\sigma

\let\gO\Omega
\let\ga\alpha

\let\gb\beta

\let\gs\sigma
\def\inpro#1{\langle#1\rangle}

\def\Im{\mathop{\rm Im}\nolimits}

\def\ker{\mathop{\rm ker}\nolimits}
\def\GL{\mathop{\it GL}\nolimits}

\def\Iff{\hskip1em\Longleftrightarrow\hskip1em}
\def\Implies{\hskip1em\Longrightarrow\hskip1em}

\def\formeq{\the\sectionno.\the\equationno}  %% equation numbering
\def\elabel#1/#2/#3/{\global\advance\equationno by 1 %
\ifx#1\empty\else\emember#1%
\ifshowlabel\marginal{\string#1}\fi\fi%
\ifmmode\eqno{#3(\formeq#2)}\else#3\formeq#2\fi} %<------ switch to \leqno ???

\def\makeblanksquare#1#2{
\dimen0=#1pt\advance\dimen0 by -#2pt
      \vrule height#1pt width#2pt depth0pt\kern-#2pt
      \vrule height#1pt width#1pt depth-\dimen0 \kern-#1pt
      \vrule height#2pt width#1pt depth0pt \kern-#2pt
      \vrule height#1pt width#2pt depth0pt
}
\font\bfbigtype=cmbx10 scaled\magstep2

\title{\bf 
Tight frames over the quaternions
and equiangular lines
}

\author{Shayne Waldron\\
 \\
Department of Mathematics \\ University of Auckland\\
Private
Bag 92019, Auckland, New Zealand\\
e--mail: waldron@math.auckland.ac.nz}

\begin{document}

\maketitle 

\begin{abstract}
We show that much of the theory of finite tight frames can be generalised 
to vector spaces over the quaternions. This includes the variational characterisation, 
group frames, and the characterisations of projective and unitary equivalence.
We are particularly interested in sets of equiangular lines (equi-isoclinic subspaces)
%(and, more generally, equi-isoclinic subspaces), 
and the groups associated with them,
and how to move them between the spaces $\Rd$, $\Cd$ and $\Hd$.
We discuss what the analogue of Zauner's conjecture for equiangular lines
in $\Hd$ might be.

%By using invariant theory, we are abl
\end{abstract}

\bigskip
\vfill


 \ifdraft (Other designs, configurations) \else\fi
15B33, \ifdraft (Matrices over special rings (quaternions, finite fields, etc.) \else\fi
%51M05, \ifdraft (Euclidean geometries (general) and generalizations) \else\fi
20G20, \ifdraft Linear algebraic groups over the reals, the complexes, the quaternions \else\fi
42C15, \ifdraft General harmonic expansions, frames  \else\fi
51M20, \ifdraft (Polyhedra and polytopes; regular figures, division of spaces [See also 51F15]) \else\fi
\quad
secondary
15B57, \ifdraft (Hermitian, skew-Hermitian, and related matrices) \else\fi
20C25, \ifdraft Projective representations and multipliers \else\fi
51M15, \ifdraft (Geometric constructions in real or complex geometry) \else\fi
65D30, \ifdraft (Numerical integration) \else\fi
94A12. \ifdraft (Signal theory [characterization, reconstruction, etc.]) \else\fi

\vskip .5 truecm
\hrule
\newpage

\section{Introduction}

Tight frames are a notion of redundant orthonormal bases which is of
both theoretical and practical interest \cite{W18}. Their recent development
has been driven by connections with algebraic combinatorics %(spherical designs)
and applications to quantum physics, signal analysis and engineering.
In all of these settings, tight frames for which the vectors/lines are ``well spread out''
are desired, with equiangular tight frames being of the most interest. 

We consider tight frames over the quaternions, motivated by equiangular
tight frames in $\Rd$ and $\Cd$. Given enough care, much of the theory generalises to quaternionic
Hilbert space $\Hd$, including the variational characterisation, group frames
and $G$-matrices, and the characterisation of projective and unitary equivalence.
We consider in detail how to move between tight frames (and associated linear 
operators) in $\Rd$, $\Cd$ and $\Hd$. 

The maximum possible number of equiangular lines in $\Rd$ is ${1\over2}d(d+1)$,
and in $\Cd$ it is $d^2$. The bound for real equiangular lines is rarely met,
but for complex lines the bound is conjectured to hold in all cases: Zauner's conjecture
on the existence of Weyl-Heisenberg SICs \cite{Zau10}, \cite{ACFW18}. 
For $\Hd$ the bound is $2d^2-d$,
for a maximum of six equiangular lines in $\HH^2$, and fifteen in $\HH^3$. We give an elementary construction
of five equiangular lines in $\HH^2$, and investigate the maximal configuration of
%six equiangular lines in $\HH^2$ recently obtained by \cite{B20}. 
six equiangular lines in $\HH^2$ recently obtained independently by
\cite{K08} and \cite{B20}. Recently, the existence of fifteen equiangular 
lines in $\HH^3$, viewed as a simplex in the projective space $\HH\PP^2$, has been 
proved by \cite{CKM16} using a Newton-Kantorovich theorem.
Based on these two data points, and my instincts (there is a lot of space in $\Hd$ and the 
beauty of the quaternions), I had initially thought the quaternionic version of
Zauner's conjecture:
%single data point (and the beauty of the quaternions): 
$$ \hbox{\em There exists $2d^2-d$ equiangular lines in $\Hd$, for each $d$}, $$
should hold. However, calculations of \cite{CKM16}, suggests that this fails for $d=4$,
and the analogous situation for the octonians is much worse. 
Thus, it seems that equiangular lines in $\CC^d$ may be a high point 
for satisfying the estimates on the maximal number of equiangular lines, 
with real and quaternionic equiangular lines 
rarely meeting the bound (``filling up all the space'') due to algebraic limitations 
of the field involved, i.e., $\RR$ not being algebraically closed and $\HH$
not being commutative.
Still, there is much interest in the maximal sets of equiangular lines
in $\Hd$, and for those in $\Rd$ (which have been studied for over half a 
century). Therefore, I present the following conjecture, which can play the role of 
Zauner's conjecture for the theory of quaternionic equiangular lines:

\begin{conjecture}
\label{HlinesConjectI}
There exists more than $d^2$ tight equiangular lines in $\Hd$, for each $d\ge 2$.
\end{conjecture}

\noindent
We observe that:
\begin{itemize}
\item
This says ``there is something going on'', i.e., there are interesting equiangular lines in $\Hd$
(ones which cannot be viewed as lines in $\Cd$) for every dimension $d$. 
\item This conjecture is known to hold only for $d=2,3$, and is otherwise open.
%\item The conjecture is known to hold only for $d=2,3$, and is otherwise open.
\item For some $d$, there do exist sets of $\le d^2$ tight quaternionic equiangular lines, e.g., 
five equiangular lines in $\HH^3$ (Example \ref{fivetightH^3}) and
six in $\HH^4$ (Example \ref{sixtightH^4}).
\item It is conjectured in \cite{CKM16} (Conjecture 4.2) that asymptotically 
there exists $N \ge (4-\sqrt{2})d$ tight equiangular 
	lines in $\Hd$, as $d\to\infty$.
%	This would not confirm Conjecture 1.
\end{itemize}

%This says ``there is something going on'', i.e., there are interesting equiangular lines in $\Hd$
%(ones which cannot be viewed as lines in $\Cd$) for every dimension $d$. 
%The conjecture is known to hold for $d=2,3$, and is otherwise open.

%or, at least, alert the reader to the possibility of the existence of ten to fifteen equiangular
%lines in $\HH^3$, and give some hints about this might play out.

%The rest of the paper is set out as follows. 

We now give the basic theory of inner product spaces over the quaternions,
%(which are not commutative), 
to a point where we are able to define and discuss tight frames over $\HH$. 
%We then develop the theory of tight frames over $\HH$.
%, introducing further properties of quaternionic spaces, as required.

\subsection{Inner products over the quaternions}

The reader is assumed to be familiar with the {\bf quaternions} $\HH$ which are an 
extension of the complex numbers $x+iy$ to a noncommutative associative algebra over the real numbers
(skew field) consisting of elements:
$$ q
=q_1+ q_2i +q_3 j +q_4 k 
=(q_1+ q_2i) +(q_3  +q_4 i)j
\in \HH, \qquad q_j\in\RR, $$
with the (noncommutative) multiplication given by Hamilton's famous formula that
$i^2=j^2=k^2=ijk=-1$.
% , which implies
%$$ ij=k, \quad jk=i,\quad ki=j, \qquad ji=-k,\quad kj=-i,\quad ik=-j. $$

Since the multiplication is not commutative, we must distinguish between left and 
right vector spaces (modules) over $\HH$. Since we wish to appropriate much 
of matrix theory, we take our vector spaces to be right $\HH$-vector spaces.
Thus $\HH$-linear maps $L$ have the form
$$ L(v_1\ga_1+\cdots+v_n\ga_n) = L(v_1)\ga_1+\cdots+L(v_n)\ga_n, $$
and can be represented by matrices, with the usual rules for multiplication,
i.e.,
$$ (AB)_{jk}=\sum_\ell a_{j\ell}b_{\ell k}, $$
where order of multiplication in $a_{j\ell}b_{\ell k}$ cannot be reversed
(see \cite{Z97}).
For those who may have noticed, I apologise for using $j$ and $k$ above as indices
for matrix entries, and elsewhere as quaternion units (as is often done with 
the complex unit $i$). 

The {\bf conjugate} and {\bf norm} of a quaternion $q=q_1+ q_2i +q_3 j +q_4 k\in\HH$ 
$$ \overline{q}:=q_1-q_2i-q_3j-q_4k, \qquad
 |q|:=\sqrt{q\overline{q}}=\sqrt{q_1^2+q_2^2+q_3^2+q_4^2}, $$
generalise the conjugate and modulus of a complex number $x+iy$, and 
allow the inner product (and associated norm) to be extended to $\HH$ as follows.
We note that
$$ \overline{ab} = \overline{b}\, \overline{a}, \qquad a,b\in\HH
\Implies (AB)^*=B^*A^*\quad\hbox{(for matrices over $\HH$)}. $$
%$$ \overline{(ab)} = (\overline{b})\, \overline{(a)}, \qquad a,b\in\HH. $$

\begin{definition} Let $\cV$ be a finite-dimensional (right) vector space 
over $\FF=\RR,\CC,\HH$. %, which is a right vector space for $\FF=\HH$.
Then an $\FF$-valued map $\inpro{\cdot,\cdot}:\cV\times \cV\to\FF$ is 
called an {\bf inner product} if it satisfies
\begin{enumerate}
\item Conjugate symmetry: $\inpro{v,w}=\overline{\inpro{w,v}}$.
%\item Linearity in the first variable:
%$\inpro{v+w,u}=\inpro{v,u}+\inpro{w,u}$,
%$\inpro{v\ga,w}=\inpro{v,w}\ga$,
% \item Linearity in the second variable:
% $\inpro{v,w\gb}=\inpro{v,w}\gb$, \\ 
% \hbox{\hskip5.5truecm $\inpro{u,v+w}=\inpro{u,v}+\inpro{u,w}$, }
\item Linearity in the second variable: $\inpro{v,w\gb}=\inpro{v,w}\gb$, $\inpro{u,v+w}=\inpro{u,v}+\inpro{u,w}$. 
\item Positive definiteness: $\inpro{v,v}>0$, $v\ne 0$.
\end{enumerate}
%$\inpro{v,v}\ge0$, with $\inpro{v,v}=0$ if and only if $v=0$.
%$\inpro{v,v}\ge0$, with $\inpro{v,v}=0$ if and only if $v=0$.
%\item $\inpro{v\ga,w}=\inpro{v,w}\ga$ (linearity in first variable) \qquad
%\hbox{(linearity in the first variable)}, $$
for all vectors $v,w,u\in \cV$ and scalars $\gb\in\FF$.
\end{definition}

\noindent
We will say that $\cV$ is a real, complex or quaternionic inner product space 
(respectively).
The theory of inner product spaces evolves as in the real and complex cases, 
though it is not well known, e.g., the Cauchy-Schwarz inequality
\begin{equation}
\label{CauchySchwarz}
|\inpro{v,w}|\le\norm{v}\norm{w}, \qquad \norm{v}:=\sqrt{\inpro{v,v}},
\end{equation}
holds (with equality if and only if $v$ and $w$ are linearly dependent), 
though it is not mentioned in the monograph \cite{R14}. I think this is in part 
due to the fact that real and complex valued inner products are often also defined on 
$\HH$-vector spaces. A good treatment is given in \cite{C80} (``unitary inner products'') 
and  \cite{GMP13} (``Hermitian quaternionic scalar products'', which includes Cauchy-Schwarz). 
The prototype of such an inner product is the {\bf Euclidean}
(or {\bf standard}) inner product
%$$ \inpro{v,w}:=\sum_j \overline{w_j}v_j, \qquad v,w\in\Hd. $$
\begin{equation}
\label{Euclideaninprodef}
\inpro{v,w}:=v^* w=\sum_j \overline{v_j}w_j, \qquad v,w\in\Hd.
\end{equation}
Throughout, we will use the notation $\inpro{v,w}$ for the Euclidean inner product,
sometimes writing $\inpro{v,w}_\FF$ to emphasize when 
all the entries of vectors $v$ and $w$ 
are in $\FF=\RR,\CC,\HH$.
The Euclidean inner product on the entries of a matrix is
the {\bf Frobenius inner product}
\begin{equation}
\label{Frobinpro}
\inpro{A,B}_F:=\trace(A^*B), \qquad \norm{A}_F^2 
=\inpro{A,A}_F=\sum_j\sum_k |a_{jk}|^2.
\end{equation}
In light of the noncommutativity of the quaternions, we note that scalars come
outside an inner product (as we have defined it) as follows
\begin{equation}
\label{inprohomogeniety}
\inpro{v\ga, w\gb} = \overline{\ga}\inpro{v,w}\gb.
\end{equation}
The notion of orthogonality, and the Gram-Schmidt process extends in the obvious fashion.
There is no need for notions of ``left'' and ``right'' orthogonality, since
$$ \inpro{v,w}=0 \Iff \inpro{w,v}=\overline{\inpro{v,w}}=0. $$
The Riesz representation also extends to inner products over $\HH$, and so the 
{\bf adjoint} of a linear map $T:\cV\to \cW$ between finite-dimensional inner product spaces
can be defined as the unique linear map $T^*:\cW\to \cV$ satisfying
%$$ \inpro{Tv,w}=\inpro{v,T^*w}, \qquad \forall v\in \cV, \ w\in \cW. $$
$$ \inpro{T^*w,v}=\inpro{w,Tv}, \qquad \forall v\in \cV, \ w\in \cW. $$
If $T$ and $T^*$ are represented as matrices $[T]$ and $[T^*]$ with respect to orthonormal bases 
$(v_j)$ and $(w_k)$, so that
$v=\sum_j v_j\inpro{v_j,v}$, $\forall v\in \cV$, 
and $w=\sum_k w_k\inpro{w_k,w}$, $\forall w\in \cW$, 
then 
$$ [T]_{jk} 
= \inpro{ w_j , T v_k }
= \inpro{ T^* w_j , v_k }
= \overline{\inpro{ v_k , T^* w_j }} 
= \overline{ [T^*]_{kj}}, $$
and hence the matrix $[T^*]$ is the conjugate transpose of the matrix $[T]$.
For this reason, it is often assumed that the inner product is the standard inner product on $\Hd$,
and all calculations are done with matrices, with $A^*$ defined to be
the conjugate transpose (or {\bf Hermitian transpose}) of the matrix $A$, 
as is the case in \cite{R14}.
The adjoint and Hermitian transpose 
satisfy some (but not all) of the usual properties, including
$$ (AB)^*=B^*A^*, \qquad 
(A+B)^*=A^*+B^*, \qquad (A^*)^*=A, $$
$$ (A^*)^{-1}=(A^{-1})^*\quad \hbox{(for $A$ invertible)}.$$
%We have not yet needed the inverse of a matrix (or linear map) over $\HH$.
Note that the transpose does not satisfy $(AB)^T=B^TA^T$ (since $\HH$ is not commutative).
It can be shown that if $AB=I$ for square matrices over $\HH$ then
$BA=I$, and so a right inverse exists for $A$ if and only if a left inverse exists, 
and these inverses are equal (and denoted by $A^{-1}$).

One subtle point, which is not obvious from the matrix formulation, is that 
scalar multiplication by $\gb\in\HH\setminus\RR$, i.e., $R_\gb:\cV\to \cV:v\mapsto v\gb$ is not an
$\HH$-linear map, since
$$ R_\gb (v\ga) = (v\ga)\gb
= v(\ga\gb)\ne v(\gb\ga)
=(v\gb)\ga=(R_\gb v)\ga \qquad
\hbox{(in general)}. $$
Left multiplication of $\Hd$ by $\gb$ defines an $\HH$-linear map $L_\gb:\Hd\to\Hd:v\mapsto \gb v$,
but this is dependent on a choice of basis: it is the linear map which maps
$e_j\mapsto e_j\gb$, i.e., the linear map whose matrix representation
with respect to the standard basis $(e_j)$ is $\gb I$ 
(see the discussion of \cite{GMP13} \S 3.1).
On the other hand,
multiplication of a fixed vector $v\in \cV$ by scalars, 
i.e.,  $[v]:\HH\to \cV:\ga\mapsto v\ga$ 
is an $\HH$-linear map:
$$ [v](\gb_1\ga_1+\gb_2\ga_2)
= v (\gb_1\ga_1+\gb_2\ga_2)
= (v\gb_1)\ga_1+(v \gb_2)\ga_2
= ([v]\gb_1)\ga_1+([v]\gb_2)\ga_2. $$
Its adjoint $[v]^*:\cV\to\HH$ is given by $[v]^*=\inpro{v,\cdot}$, since
$$ \inpro{ \ga , [v]^*w } 
= \inpro{[v]\ga,w}
=\inpro{v\ga,w}
=\overline{\ga}\inpro{v,w}
=\overline{\ga}\inpro{1,\inpro{v,w}}
= \inpro{\ga,\inpro{v,w}}. $$
The map $[v]$ is sometimes abbreviated simply as $v$, especially 
when $v\in\Hd$ is thought of as a column vector, i.e., 
as an element of $\HH^{d\times 1}$.
More generally, a {\bf synthesis map}
$$ V=[v_1,\ldots,v_n]:\Hn\to\cV:a\mapsto v_1a_1+\cdots+v_n a_n, $$
for a sequence of vectors $v_1,\ldots,v_n\in\cV$, has adjoint 
the {\bf analysis map}
$$ V^*:\cV\to\Hn: v\mapsto(\inpro{v_j,v})_{j=1}^n. $$

\section{Tight frames}

A {\bf frame} for a Hilbert space $\cH$ is a sequence of vectors $(v_j)$ 
satisfying the 
%(somewhat cryptic, but easier to verify in the infinite-dimensional setting) 
condition
\begin{equation}
\label{framedefn}
A\norm{v}^2\le\sum_j|\inpro{v_j,v}|^2\le B\norm{v}^2, \qquad\forall v\in\cH,
\end{equation}
where $A,B>0$ are constants.
%, with the case $A=B$ giving a {\bf tight frame}.
From this, a ``frame expansion'' follows, which takes a particularly simple form
when $A=B$, i.e.,
$$  v= {1\over A}\sum_j v_j \inpro{v_j,v}, \qquad\forall v\in\cH.  $$
A prominent early example of the use of such ``generalised orthonormal bases'' is
in the theory of wavelets. 
Recently, frames have been considered for quaternionic
Hilbert space, see, e.g., \cite{KTS17}, \cite{VSS20} (which deal primarily
with the frame operator and the construction of dual frames), 
\cite{KPS21} (equiangular lines and Hadamard matrices),
and
\cite{NS22} which considers the connectedness of the algebraic variety 
of quaternionic frames with given norms.
Here we consider
tight frames (where the dual frame is the frame itself) with a particular emphasis
on the classification and construction of such frames. This is related to earlier 
work of Hoggar \cite{H76}, \cite{H82} and others, which implicitly considers tight frames 
over quaternionic (and even octonionic) Hilbert spaces.

\subsection{Tight frames defined and unitary equivalence}

We will say that a sequence of vectors with synthesis map $V=[v_1,\ldots,v_n]$ 
is a {\bf tight frame} for a (finite-dimensional) quaternionic Hilbert space
$\cH$ if it satisfies (\ref{framedefn}), where $A=B$, and is {\bf 
normalised} if $A=B=1$, which can be achieved by multiplying the vectors of a tight frame by 
a suitable positive scalar.  
The {\bf frame operator} (for a sequence of vectors) is $S=VV^*$ and
the {\bf Gramian} (matrix) is $G=V^*V$. 

A linear map $U$ on $\cH$ is unitary if it preserves angles,
i.e., $\inpro{Uv,Uw}=\inpro{v,w}$, $\forall v,w$, or, equivalently 
$U^*U=I$. Unitary maps can be defined in the same way on quaternionic
Hilbert spaces. If $V=[v_1,\ldots,v_n]$ is a frame for a 
quaternionic Hilbert space, then so is any unitary image
$UV=[Uv_1,\ldots,Uv_n]$, and these frames have the same
Gramian since $(UV)^*UV=V^*U^*UV=V^*V$, and we say that they are 
{\bf unitarily equivalent}. Tight frames are studied up to unitary equivalence
(which is an equivalence relation) and multiplication by a nonzero scalar.

The monograph \cite{W18} is a good reference for those parts of
the theory of finite tight frames which we now develop.
First we consider equivalent conditions for being a tight frame.
For this, we need the polarisation identity for quaternionic Hilbert space.
%Since this is not well known (it is not in \cite{R14}), we provide it with proof.
Since this is not well known, we provide it with proof.

%\begin{lemma}
%\label{polarisationidlemma}
%(Polarisation identity) For an inner product space over
%$\FF=\RR,\CC,\HH$,
%%$$ \inpro{v,w}:=\sum_j \overline{w_j}v_j, $$
%we have
%	$$ \inpro{v,w} = {1\over 4} \sum_{r=0}^{m-1} \Bigl(\norm{v+wi_r}^2-\norm{v-wi_r}^2\Bigr)i_r, $$
%where $m=\dim_\RR(\FF)$, $(i_0,i_1,i_2,i_3)=(1,i,j,k)$,
%and $\inpro{\cdot,\cdot}$ is linear in the first variable.
%\end{lemma}
%
%\begin{proof}
%We first observe %(by a calculation) 
%that for a quaternion
%$q =q_0+q_1i_1+q_2i_2+q_3i_3$, $q_r\in\RR$, a calculation gives
%\begin{equation}
%\label{qrpart}
%\overline{i_r}q+\overline{q}i_r = 2 q_r, \qquad r=0,1,2,3,
%\end{equation}
%and we write $(q)_r=q_r$. Expanding, using the properties of the inner product, gives
%\begin{align*} \norm{v\pm w i_r}^2
%&=\inpro{v,v}+\inpro{\pm wi_r,\pm wi_r}
%+\inpro{v,\pm wi_r} +\inpro{\pm wi_r,v} \cr
%&=\norm{v}^2+\norm{w}^2
%\pm\overline{i_r} \inpro{v, w} \pm \inpro{w,v} i_r,
%\end{align*}
%so that
%$$ \norm{v+w i_r}^2-\norm{v-w i_r}^2
%= 2\bigl(\overline{i_r} \inpro{v, w} + \overline{\inpro{v,w}} i_r\bigr)
%= 4(\inpro{v,w})_r, $$
%which gives the result.
%\end{proof}

\begin{lemma}
\label{polarisationidlemma}
(Polarisation identity) For an inner product space over
$\FF=\RR,\CC,\HH$,
%$$ \inpro{v,w}:=\sum_j \overline{w_j}v_j, $$
we have
$$ \inpro{v,w} = {1\over 4} \sum_{r=0}^{m-1} \Bigl(\norm{vi_r+w}^2-\norm{vi_r-w}^2\Bigr)i_r, $$
where $m=\dim_\RR(\FF)$, $(i_0,i_1,i_2,i_3)=(1,i,j,k)$,
and $\inpro{\cdot,\cdot}$ is linear in the second variable.
\end{lemma}

\begin{proof}
We first observe %(by a calculation) 
that for a quaternion
$q =q_0+q_1i_1+q_2i_2+q_3i_3$, $q_r\in\RR$, a calculation gives
\begin{equation}
\label{qrpart}
\overline{i_r}q+\overline{q}i_r = 2 q_r, \qquad r=0,1,2,3,
\end{equation}
and we write $(q)_r=q_r$. Expanding, using the properties of the inner product, gives
\begin{align*} \norm{vi_r\pm w}^2
&=\inpro{vi_r,vi_r}+\inpro{\pm w,\pm w}
	+\inpro{vi_r,\pm w}
+\inpro{\pm w,vi_r} 
	\cr
&=\norm{v}^2+\norm{w}^2
\pm\overline{i_r} \inpro{v, w} \pm \inpro{w,v} i_r,
\end{align*}
so that
$$ \norm{vi_r+w}^2-\norm{vi_r-w}^2
= 2\bigl(\overline{i_r} \inpro{v, w} + \overline{\inpro{v,w}} i_r\bigr)
= 4(\inpro{v,w})_r, $$
which gives the result.
\end{proof}

This could also be proved by rewriting equation (3.5.1) of \cite{R14}
for $A=I$ and $(q_1,q_2,q_3)=(i,j,k)$.

The basic characterisations of normalised tight frames generalise.

\begin{proposition}
\label{tightframeequidefs}
Let $V=[v_1,\ldots,v_n]$ be sequence of vectors in a $d$-dimensional
(right) quaternionic Hilbert space $\cH$, such as $\Hd$. Then the following are
equivalent
\begin{enumerate}[\rm(i)]
\item $V$ is a normalised tight frame for $\cH$, i.e.,
$$ \norm{v}^2= \sum_j |\inpro{v_j,v}|^2, \qquad\forall v\in\cH. $$
\item The frame operator $S=VV^*=I$, i.e., we have the frame expansion
$$ v = \sum_j v_j \inpro{v_j,v}, \qquad\forall v\in\cH. $$
\item The Plancherel identity
$$ \inpro{v,w} = \sum_j \inpro{v,v_j} \inpro{v_j,w}, \qquad\forall v,w\in\cH. $$
\item The Gramian $P=V^*V$ is a rank $d$ orthogonal projection, i.e., $P^2=P$, $P^*=P$.
\end{enumerate}
\end{proposition}

\begin{proof} The implications (ii)$\implies$(iii)$\implies$(i) follow by taking
the inner product with $w$ and then letting $w=v$, respectively. Suppose that (i) holds.
	By Lemma \ref{polarisationidlemma} and (\ref{qrpart}),
we have
\begin{align*}
4(\inpro{v,w})_r
&= \norm{vi_r+w}^2-\norm{vi_r-w}^2
= \sum_j\Bigl( |\inpro{v_j,vi_r+w}|^2 -|\inpro{v_j,vi_r-w}|^2\Bigr) \cr
&= \sum_j \Bigl( 2\overline{i_r}\inpro{v,v_j}\inpro{v_j,w}
+2\inpro{w,v_j}\inpro{v_j,v}i_r\Bigr)
= 4 \sum_j( \inpro{v,v_j} \inpro{v_j,w}
	)_r.
\end{align*}
Thus (by the Riesz representation, or since the orthogonal complement of $\Hd$ is $\{0\}$)
$$ \inpro{v,w}
	=\sum_j \overline{\inpro{v_j,v}} \inpro{v_j,w}
= \inpro{ \sum_j v_j \inpro{v_j,v} ,w} \Implies
 v = \sum_j v_j \inpro{v_j,v}, $$
which is (ii).

We now show
% For the implications
 (iii)$\iff$(iv).
We observe that by construction $P=(\inpro{v_j,v_k})_{j,k}$ is Hermitian. % and of rank $d$.
The condition $P^2=P$ can be written entrywise as
$$ \inpro{v_j,v_k}=P_{jk}=\sum_\ell P_{j\ell}P_{\ell k}
=\sum_\ell \inpro{v_j,v_\ell}\inpro{v_\ell,v_k}, $$
which is the Plancherel identity for $v=v_j$ and $w=v_k$. The implications
then follow by extending the Plancherel identity (using linearity and symmetry
of the inner product), and
calculating
%determining the rank of $P$ via
$\rank(P)=\trace(P)=\Re(\trace(VV^*))=d$, by (ii).
\end{proof}

%\begin{proposition}
%\label{PGramianchar}
%A sequence of vectors $V=[v_1,\ldots,v_n]$ in $\Fd$ is a tight frame for $\Fd$
%if and only if its scaled Gramian $P={1\over A}V^*V$, $dA:=\sum_j\norm{v_j}^2$
%is an orthogonal
%projection matrix, i.e., 
%$$ (V^*V)^2=A\, V^*V. $$
%\end{proposition}

For ease of presentation, we will now consider $\Hd$,
rather than saying let $\cH$ be a quaternionic Hilbert space
of dimension $d$. 
We also write $\Fd$, with $\FF=\RR,\CC,\HH$.
The following characterisation extends the real and complex cases (see \cite{W18}
Theorem 2.1).

\begin{proposition} An $n\times n$ matrix $P$ is the Gramian matrix of a normalised tight
frame $V=[v_1,\ldots,v_n]$ for $\Hd$ if and only if it is an orthogonal 
projection matrix of rank $d$.
\end{proposition}

\begin{proof}
We have already seen that a normalised tight frame is determined by its Gramian,
which is an orthogonal projection of rank $d$ (Proposition \ref{tightframeequidefs}).
It remains only to show that such a matrix $P$ corresponds to a normalised tight 
frame. Let $v_j=Pe_j$. Then with the Euclidean norm on $\Hn$, we have that
%$$ \inpro{v_k,v_j} = \inpro{Pe_k,Pe_k} =\inpro{e_k,Pe_j}=\overline{P_{kj}}=P_{jk}, $$
$$ \inpro{v_j,v_k} = \inpro{Pe_j,Pe_k} =\inpro{e_j,Pe_k}=P_{jk}, $$
so that $(v_j)$ is such a tight frame (for its $d$-dimensional span).
\end{proof}

A finite sequence of unit vectors $(v_j)$ (or the lines they represent) are said 
to be {\bf equiangular} if 
\begin{equation}
\label{equiangulardefn}
	|\inpro{v_j,v_k}|^2=\gl=c^2=(\cos\gth)^2, \qquad\forall j\ne k.
\end{equation}
The constants $\gl$, $c$ and $\gth$ all occur in the literature, 
%and we will refer to them as the (common) angle. 
and are called the (common) angle. 

%$$ \overline{i}q+\overline{q}i = \overline{q}i-iq
%= (a-bi-cj-dk)i -i(a+bi+cj+dk) = (ai+b+ck-dj) +(-ai +b -ck+dj) = 2b. $$
%$$ \overline{j}q+\overline{q}j = \overline{q}j-jq
%= (a-bi-cj-dk)j -j(a+bi+cj+dk) = (aj-bk+c+di) +(-aj +bk +ck-di) = 2b. $$

\begin{example}
\label{HoggarlinesinC^2}
Four equiangular lines in $\HH^2$ 
with $\gl={1\over3}$
%at angle $\gl={1\over3}$
are given in \cite{H76}, % (Example 1.5), 
namely
\begin{align*}
w_1 &= {1\over\sqrt{2}}\pmat{1\cr j},\quad 
w_2 = {1\over\sqrt{6}}\pmat{1-\sqrt{2}i\cr j-\sqrt{2}k}, \cr
w_3 &= {1\over2\sqrt{3}} \pmat{\sqrt{2}+\sqrt{3}+i\cr\sqrt{2}j-\sqrt{3}j+k}, \quad
w_4 = {1\over2\sqrt{3}} \pmat{\sqrt{2}-\sqrt{3}+i\cr\sqrt{2}j+\sqrt{3}j+k}.
\end{align*}
The Gramian of these vectors (which are a tight frame for $\HH^2$) 
has only complex entries, and so they are unitarily equivalent 
to an equiangular tight frame for $\CC^2$.
They have the same Gramian as the Weyl-Heisenberg SIC 
$v_1=v$, $v_2=Sv$, $v_3=\gO v$, $v_4=iS\gO v$, where
$$ v={1\over\sqrt{6}}\pmat{ \sqrt{3+\sqrt{3}}\cr 
%e^{{\pi\over 4} i} 
{1\over\sqrt{2}}(1+i)
\sqrt{3-\sqrt{3}} },  \quad
S=\pmat{0&1\cr1&0}, \ \gO=\pmat{1&0\cr0&-1}. $$
Therefore, there is a unitary map $U$ with $v_j=Uw_j$, which we calculate as
%The $m$-products are all complex numbers, which suggests this could lie
%in $\CC^2$. We seek a unitary map $U$ with $v_j:=(Uw_j)\ga_j$. 
%We find $v_1=Uw_1$, $v_2=Uw_2$, $v_3=Uw_3$, $v_4=-iUw_4$, where
$$ U=\pmat{z_1&-jz_1\cr z_2&-kz_2},
\quad
z_1:=
{\sqrt{3+\sqrt{3}}\over2\sqrt{3}}
+{\sqrt{3-\sqrt{3}}\over2\sqrt{3}}i, \quad 
z_2:={\sqrt{3+\sqrt{6}}\over2\sqrt{3}}
-{\sqrt{3-\sqrt{6}}\over2\sqrt{3}}i.
$$
%$$ \inpro{v_1,v_2}
%=\inpro{(Uw_1)\ga_1,(Uw_2)\ga_2} 
%=\overline{\ga_2} \inpro{w_1,w_2} \ga_1
%\Implies \ga_1=1, \quad \ga_2=\inpro{w_1,w_2}\inpro{v_1,v_2}^{-1}, $$
%$$ U[w_1\ga_1,w_2\ga_2]=[v_1,v_2] \Implies
%U=[v_1,v_2] [w_1\ga_1,w_2\ga_2]^{-1}. $$
\end{example}

Though this first example of quaternionic equiangular lines are not
``quaternionic'', we will see that such lines do exist, and they are very intriguing.

\subsection{The variational characterisation of tight frames}

We now seek to extend the variational characterisation 
for tight frames \cite{BF03}, \cite{W03}.
For $\Cd$, this is most easily and transparently
proved % over $\CC$ 
from 
the spectral decomposition of the frame operator using 
%the formula 
$\trace(AB)=\trace(BA)$ (see \cite{W18}, Theorem 6.1).
This trace formula no longer holds over the quaternions, even for $1\times 1$ matrices. 
Instead, we will use the fact 
%that if $U$ is unitary and $\gL$ is a real diagonal matrix, then 
%$$  \trace(U^*\gL U)=\trace(\gL), $$
\begin{equation}
\label{Htraceformulareal}
\Re(\trace(AB))=\Re(\trace(BA)),
\end{equation}
which follows from the special case 
$\Re(ab)=\Re(ba)$, $\forall a,b\in\HH$. 
%$$ \Re(ab)=\Re(ba), \qquad\forall a,b\in\HH. $$

The general spectral theory of matrices over $\HH$ is fraught (see \cite{R14}), since
%The spectral theory is complicated by the fact that if 
$$ A v = v\gl \Implies A(v\ga)=(v\ga) \ga^{-1}\gl\ga, $$
so that if $v$ is a (right) eigenvector for $\gl$, then $v\ga$ is an eigenvector for eigenvalue $\ga^{-1}\gl\ga$.
However, {\bf Hermitian matrices} (those with $A^*=A$), have real eigenvalues and 
are unitarily diagonalisable, as in the complex case.

%\begin{equation}
%\label{Htraceformuls}
%\trace(U^*\gL U)=\trace(\gL), \qquad 
%\hbox{for $U$ unitary and $\gL$ real diagonal}. 
%\end{equation}
%to give a modified argument which also works over the quaternions.

\begin{lemma} Let $V=[v_1,\ldots,v_n]$ be %a sequence of 
vectors in $\Fd$, with 
frame operator $S=VV^*$ and Gramian $G=V^*V$. Then $\trace(S^k)=\trace(G^k)$,
$k=1,2,\ldots$. In particular,
%$$ S=VV^* \quad\hbox{(frame operator)}, \qquad
%G=V^*V, \quad {(Gramian)}, \qquad V=[v_1,\ldots,v_n]. $$
\begin{equation}
\label{traceSandSsquared}
\trace(S) %= \norm{S^{1\over2}}_F^2 
= \sum_{j} \norm{v_j}^2, 
\qquad \trace(S^2)  %= \norm{S}_F^2 
= \sum_{j}\sum_{k} |\inpro{v_j,v_k}|^2.
\end{equation}
\end{lemma}

\begin{proof} The trace of an Hermitian matrix $A$ is real,
since $\overline{\inpro{Ax,x}}=\inpro{x,Ax}=\inpro{Ax,x}$.
Since $S^k$ and $G^k$ are Hermitian, they have real trace, and so
by (\ref{Htraceformulareal}), we have
\begin{align*}
\trace(S^k) 
&= \Re(\trace(VV^*(VV^*)^{k-1}))
= \Re(\trace(V^*(VV^*)^{k-1}V)) \cr
&= \Re(\trace((V^*V)^k))
= \trace(G^k).
\end{align*}
%We first prove (\ref{Htraceformuls}), by direct computation
%\begin{align*}
%\trace(U^*\gL U)
%&= \sum_j \sum_{r,s} (U^*)_{jr}(\gL)_{rs}(U)_{sj}
%= \sum_j \sum_{r,s} \overline{u_{rj}} (\gL)_{rs} u_{sj} \cr
%&= \sum_{r,s} (\gL)_{rs} \sum_j \overline{u_{rj}} u_{sj} 
%= \sum_{r,s} (\gL)_{rs} \gd_{rs}
%= \sum_{r} (\gL)_{rr} 
%= \trace(\gL). 
%\end{align*}
%Let $V=U_1\Sigma U_2^*$ be a singular value decomposition 
%for $V$ (see \cite{R14}, Proposition 3.2.5). 
%The singular values of $V$ are the nonzero diagonal entries
%of $\Sigma$. 
%Since $U_1$ and $U_2$ are unitary, we have
%$$ S^k=(VV^*)^k=(U_1\Sigma U_2^* U_2 \Sigma^T U_1^*)^k=U_1(\Sigma\Sigma^T)^k U_1^*, $$
%$$ G^k=(V^*V)^k
%= ( U_2 \Sigma^T U_1^* U_1\Sigma U_2^* )^k
%= U_2 (\Sigma^T\Sigma)^k U_2^*, $$
%where $\Sigma\Sigma^T$ and $\Sigma^T\Sigma$ are real diagonal, 
%with the same trace. 
%Hence, 
%by (\ref{Htraceformuls}),
%$$ \trace(S^k)
%=\trace(U_1(\Sigma\Sigma^T)^k U_1^*)
%=\trace\bigl((\Sigma\Sigma^T)^k\bigr)
%=\trace\bigl((\Sigma^T\Sigma)^k\bigr)
%=\trace(U_2(\Sigma^T\Sigma)^k U_2^*)
%=\trace(G^k). $$
The formulas for $\trace(G)$ and $\trace(G^2)$ given on the 
left hand side of 
(\ref{traceSandSsquared}) 
are easily calculated
from $(G)_{jk}=\inpro{v_j,v_k}$. 
\end{proof}

\begin{theorem}
\label{generalisedWelchbound}
(Variational characterisation)
Let $v_1,\ldots,v_n$ be vectors in $\Fd$,
which are 
not all zero.  %, and $d=\dim(\cH)$. 
Then
\begin{equation}
\label{genWelchbd}
\sum_{j=1}^n \sum_{k=1}^n |\inpro{v_j,v_k}|^2 \ge {1\over d}
\Bigl(\sum_{j=1}^n \norm{v_j}^2\Bigr)^2,
%\Bigl(\sum_{j=1}^n \inpro{f_j,f_j}\Bigr)^2, 
\end{equation}
with equality if and only if $(v_j)_{j=1}^n$ is a tight frame for $\Fd$.
\end{theorem}

\begin{proof}
Let $V=[v_j]$. 
%$$ \trace(S)
%=\sum_{k=1}^d \sum_{j=1}^n (V)_{kj}(V^*)_{jk}
%=\sum_{k=1}^d \sum_{j=1}^n v_{kj}\overline{v_{kj}}
%= \sum_{j=1}^n \sum_{k=1}^d |v_{kj}|^2
%=\sum_{j=1}^n \norm{v_j}^2. $$
Since $S=VV^*$ is positive definite,
%strictly posiis self adjoint, 
it is unitarily diagonalisable $S=U\gL U^*$, $\gL=\diag(\gl_j)$,
with real 
eigenvalues $\gl_1,\ldots,\gl_d\ge0$.
From (\ref{Htraceformulareal}), 
we have 
$$\trace(S^k)
=\Re(\trace(U\gL^k U^*))
=\Re(\trace(\gL^k U^*U))
=\Re(\trace(\gL^k))
=\trace(\gL^k). $$
Thus, the Cauchy-Schwarz inequality gives
$$ \trace(S)^2 = (\sum_j\gl_j)^2
= \inpro{(1),(\gl_j)}^2
\le \norm{(1)}^2\norm{(\gl_j)}^2
= d \sum_j \gl_j^2 = d\trace(S^2),$$
which, % by (\ref{Strace}) 
by (\ref{traceSandSsquared}),
is (\ref{genWelchbd}),
with equality if and only if $\gl_j=A$, $\forall j$, $A>0$, i.e.,
$$ S=U(AI)U^*=AI  \Iff \hbox{$(v_j)$ is a tight frame for $\Fd$.}$$
Note above, since one vector is nonzero, $S=\sum_j v_jv_j^*\ne0$, and so $A\ne0$.
%Note above, since one of the vectors $(v_j)$ is nonzero, $S\ne0$, and so $A\ne0$.
\end{proof}

This variational characterisation of tight frames depends only on the 
Gramian, and hence the frame up to unitary equivalence. It is easy to
verify, and plays a key role in Theorems \ref{tightframesRtoC}
and \ref{tightframesCtoH}. We now consider its implications for
equiangular lines.

%In addition to being simpler to verify than the definition of a tight frame,
%(\ref{genWelchbd})
%can be used to find tight frames, e.g., a tight frame of $n$ unit vectors
%can be found by minimising the left hand side (to ${n^2\over d}$).

\subsection{Bounds on equiangular lines}

We recall that unit vectors $(v_j)$ in $\Fd$ are
equiangular if they satisfy (\ref{equiangulardefn}), i.e.,
$$ |\inpro{v_j,v_k}|^2=\gl=c^2=(\cos\gth)^2,
\qquad\forall j\ne k. $$
Those of the most interest have the maximum separation of the corresponding
lines, i.e., $\gl=c^2$ {\em small}, or, equivalently, $0\le\gth\le{\pi\over2}$ {\em large}.
Examples that exist in every dimension $d$ are orthonormal bases
of $n=d$ vectors ($\gl=0$, $\gth=90^\circ$) and the $n=d+1$ vertices
of a regular simplex ($\gl={1\over d^2}$). 
The formula for the {\em chordal distance}
$$ \rho(v_j,v_k) := \sqrt{1-|\inpro{v_j,v_k}|^2}, $$
gives a metric on the lines in $\Hd$, and accordingly,
\cite{CKM16} calls sets of (tight) equiangular lines ``(tight) simplices in projective space''
(points an equal distance from each other).

As an example of Theorem
\ref{generalisedWelchbound}, we have the following bound.

\begin{example}
If all the $n$ vectors $(v_j)$ in $\Fd$ have unit norm, then
(\ref{genWelchbd})
reduces to
$$\sum_{j=1}^n \sum_{k=1}^n |\inpro{v_j,v_k}|^2 \ge {1\over d}
\Bigl(\sum_{j=1}^n 1^2\Bigr)^2={n^2\over d}. $$
Moreover, if the $(v_j)$ are equiangular, then the left hand side is
$(n^2-n)\gl+n$, and the inequality rearranges to
\begin{equation}
\label{lamdalinesboung}
\gl \ge {n-d\over d(n-1)}, 
\end{equation}
with equality (and maximum possible separation) when the vectors are a tight frame,
and for $\gl<{1\over d}$ it rearranges to the {\bf relative bound} for 
equiangular lines
$$ n \le {1-\gl\over {1\over d}-\gl}, \qquad \gl<{1\over d}. $$
\end{example}

The next bound (which is well known for $\FF=\RR,\CC$) depends
on the underlying field.

%The equiangularity condition on a set of lines ensures
%that the corresponding set
%of one-dimensional orthogonal projections is linearly independent over $\RR$,
%which gives a bound on the possible number of equiangular lines in $\Fd$
%that depends on $\FF$.

\begin{theorem}
\label{linearindepth}
Suppose $d>1$. Let $(v_j)$ be a sequence of $n$ non-parallel unit vectors in $\Fd$
%, $d>1$, 
giving a set of $n$ equiangular lines, then the orthogonal projections
$$ P_j =v_j v_j^* : v\mapsto v_j\inpro{v_j,v}, \qquad j=1,\ldots,n, $$
are linearly independent over $\RR$, and hence
\begin{equation}
\label{nequiangbound}
n \le 
\begin{cases}
 {1\over2}d(d+1), & \FF=\RR; \\
 d^2, & \FF=\CC; \\
 2d^2-d, & \FF=\HH,
\end{cases}
\end{equation}
with equality if and only if $(P_j)$ is a basis for the
$\RR$-vector space of Hermitian matrices. In these cases, the angle is 
\begin{equation}
\label{lambdamaxight}
\gl =
\begin{cases}
 {1\over d+2}, & \FF=\RR; \\
 {1\over d+1}, & \FF=\CC; \\
 {1\over d+{1\over2}}, & \FF=\HH.
\end{cases}
\end{equation}
\end{theorem}

\begin{proof} Since $d>1$, the equiangularity constant $\gl$ is less than $1$.
Using (\ref{Htraceformulareal}), we calculate
\begin{align*}
\Re(\trace(P_jP_k))
&=\Re(\trace(v_jv_j^*v_kv_k^*)) \cr
&=\Re(\trace(v_j^*v_kv_k^*v_j))
=|\inpro{v_j,v_k}|^2=\gl, \qquad j\ne k.
\end{align*}
The $\RR$-linear combination $\sum_j c_j P_j$ is Hermitian, and hence 
its Frobenius norm satisfies
\begin{align*}
\norm{\sum_j c_j P_j}_F^2
&= \Re(\trace(\sum_j c_j P_j \sum_k c_kP_k))
= \sum_j\sum_k c_j {c_k} \Re(\trace(P_jP_k)) \\
&= \sum_j\sum_k c_j {c_k} \gl + \sum_j c_j {c_j} (1-\gl)
%&= \gl \Bigl|\sum_j c_j\Bigr|^2 + (1-\gl) \sum_j |c_j|^2, \cr
= \gl \Bigl(\sum_j c_j\Bigr)^2 + (1-\gl) \sum_j c_j^2,
\end{align*}
which is zero only for the trivial linear combination.

The $n$ projections $\{P_j\}$ belong to
the real vector space of $d\times d$ Hermitian matrices
which has dimension 
given by the right hand side of (\ref{nequiangbound}).
For example, for $\FF=\HH$ the Hermitian matrices are 
determined by their real diagonal, and the entries above
it which can be any quaternions, giving a dimension
of  $d+{1\over2}(d^2-d)\cdot 4= 2d^2-d$.
\end{proof}

This result for $\Hd$, the inequality
(\ref{nequiangbound}), is given in \cite{H76b}, without proof,
and as Proposition 2.2 in \cite{CKM16} (which also includes
the octonionic case $\OO^3$).

%This result for $\Hd$, the inequality
%(\ref{nequiangbound}), is given in \cite{H76b}, without proof,
%and does not seem to be widely known.

We are now in a position to discuss quaternionic equiangular lines.
We first observe:
\begin{itemize}
\item Quaternionic equiangular lines do exist (for $\gl<1$, $d>1$). 
\end{itemize}
%{\em they do exist} (for $\gl<1$). 
You will recall from Example \ref{HoggarlinesinC^2}
that Hoggar's example of four equiangular
lines in $\HH^2$ were in fact lines in $\CC^2$ (most likely
the very first occurrence of a SIC in the literature).
For $d=1$, any sequence of unit quaternions is an equiangular tight frame
(with $\gl=1$), which is quaternionic if any ratio of the quaternions
is not a complex number. Even though this is a trivial example, we will
be able to use such frames to construct unit-norm tight frames in $\CC^2$ and $\RR^4$
%(see Section \ref{RtoCtoH}). 
(Example \ref{quaternionsgroupframe}).
We now give a simple example in $\HH^2$.

\begin{example} (Five equiangular lines in $\HH^2$).
Fix $0<t<1$, and consider the four unit vectors
$$ v_r=\pmat{t\cr\sqrt{1-t^2}\, i_r}, \qquad i_1=1, \quad i_2=i, \quad i_3=j,\quad i_4=k. $$
%$$ v_r=\pmat{t\cr\sqrt{1-t^2}i_r}, \qquad r=0,1,2,3. $$
These are equiangular, with
$$ |\inpro{v_r,v_s}|^2=\gl:=t^4+(1-t^2)^2, \qquad j\ne k, $$
where ${1\over2}\le\gl<1$. 
By Theorem \ref{linearindepth},
the maximal number of equiangular lines in 
$\CC^2$ is four, with $\gl={1\over 3}$, so these lines are quaternionic.
For the maximal separation $\gl={1\over2}$, we may add a fifth equiangular line,
to obtain five equiangular lines in $\HH^2$ given by
\begin{equation}
\label{fiveMUBlines}
{1\over\sqrt{2}}\pmat{1\cr 1}, \quad
{1\over\sqrt{2}}\pmat{1\cr i}, \quad
{1\over\sqrt{2}}\pmat{1\cr j}, \quad
{1\over\sqrt{2}}\pmat{1\cr k}, \quad 
\pmat{1\cr 0} \ \bigl(\hbox{or }
\pmat{0\cr 1}\bigr). 
\end{equation}
These lines are not tight, since they do not give equality in 
(\ref{lamdalinesboung}), i.e.,
$$ \gl={1\over 2}>{3\over 8}={5-2\over2(5-1)} = {n-d\over d(n-1)}. $$
They appear exactly as above in \cite{B20}, 
for the parameter choice
	$c={1\over\sqrt{2}}$, $\omega={\pi\over4}$, $\alpha=0$ and 
	$\gamma={\pi\over4}$.

Taking the five lines of (\ref{fiveMUBlines}) and their orthogonal complement gives
five MUBs (mutually unbiased bases) in $\HH^2$, which is a tight frame of
ten vectors (see \cite{BW25}).
\end{example}

%nnnnnnnnnnnnnnnnnnnnnnnnnnnnnnnnnnnnnnnnnnn
%\begin{example}
%We can generalise the previous example. 
%	Let $(v_j)$ be an equiangular tight frame of $n$ vectors 
%	for $\RR^{d-1}$, at angle
%	$$ |\inpro{v_j,v_k}|=c, \qquad j\ne k. $$
%Then the $4n$ vectors
%	$$ w_{j,r} = \pmat{ tv_j\cr \sqrt{1-t^2}i_r}\in\RR^{d-1}\times\HH
%	\subset\Hd, \qquad 1\le j\le n,\ 1\le r\le 4, $$
%have inner products
%	$$ \inpro{w_{j,r},w_{k,s}}
%	=t^2\inpro{v_j,v_k}+(1-t^2)\overline{i_r}i_s $$
%\end{example}
%nnnnnnnnnnnnnnnnnnnnnnnnnnnnnnnnnnnnnnnnnnn

Another method to obtain tight equiangular lines is via
the {\bf complementary tight frame}. The construction is as follows.
Let $G$ be the Gramian of $n>d$ equiangular unit vectors in $\Fd$ 
at an angle $\gl={n-d\over d(n-1)}\ne 0$, 
so that $P={d\over n} G$ is an orthogonal projection matrix
(Proposition \ref{tightframeequidefs}).
The complementary orthogonal projection $Q=I-P$ gives an equiangular tight frame 
of $n$ vectors for $\FF^{n-d}$ with Gramian $G_c$ given by
$$ %Q = {n-d\over n}G^C \Implies 
G_c={n\over n-d}I-{d\over n-d}G, $$
and common angle $\gl_c={d^2\over(n-d)^2}\gl={d\over (n-d)(n-1)}$.
The equivalent construction for lines is called the {\em Gale dual} in \cite{CKM16} 
(see Corollary 2.12).
%This is called the {\bf complementary} equiangular tight frame.

Let $c_d$ be the right hand side of (\ref{nequiangbound}),
which we can write as
$$ c_d = d+{1\over2}d(d-1)\cdot m, \qquad m:=\dim_\RR(\FF). $$
Since the complementary tight frame also must satisfy the bound (\ref{nequiangbound}),
for $n-d\ne1$,
we have that an equiangular tight frame of $n>d+1$ unit vectors in $\Fd$ must
satisfy
\begin{equation}
\label{complementbound}
n\le \min\{c_d,c_{n-d}\}. 
\end{equation}
This gives the following (see Theorem 2.18 of \cite{K08}).

\begin{proposition} 
\label{tightlinesboundofn}
Let $d\ge2$.
An equiangular tight frame of $n>d+1$ vectors for $\Fd$ satisfies
\begin{equation}
\label{lowerupperbound}
d+{1\over2}+{\sqrt{{8\over m}d+1}\over 2} \le n
\le d+{m\over2}d(d-1), \qquad m=\dim_\RR(\FF),
\end{equation}
so that
\begin{equation}
\label{nicelowerbound}
n\ge d+2+j, \qquad 
	\hbox{for }\quad d > {m\over2}(j+1)(j+2).
\end{equation}
%\hbox{for $d\ge {m\over2}(j+1)(j+2)$}. $$
%$$ \FF:\qquad d>{m\over2} j(j+1)  \Implies n \ge d+2+j. $$
\end{proposition}

\begin{proof}
The condition $n\le c_{n-d}$ in (\ref{complementbound}) % The condition $n\le c_{n-d}$ 
can be written as
%$$ n\le n-d+{1\over2}(n-d)(n-d-1)\cdot m, $$
%$$ -d+{1\over2}(n-d)(n-d-1)\cdot m \ge 0, $$
%$$ {1\over2}(n-d)(n-d-1)\cdot m -d \ge 0, $$
%$$ (n-d)(n-d-1)- {2\over m}d \ge 0, $$
$$ n^2-(2d+1)n+d(d+1)- {2\over m}d \ge 0. $$
By considering the roots of this quadratic polynomial in $n$, this is satisfied if and only if
$$ n\le d+{1\over2}-{\sqrt{{8\over m}d+1}\over 2}<d, 
\quad\hbox{or}\quad
 n\ge d+{1\over2}+{\sqrt{{8\over m}d+1}\over 2},  $$
which gives the lower bound in (\ref{lowerupperbound}). The upper bound is
the condition $n\le c_d$.
%given by (\ref{nequiangbound}).

Rearranging the right hand inequality in 
$$ n \ge d+{1\over2}+{\sqrt{{8\over m}d+1}\over 2} \ge d+2+j, $$
gives 
$$ d \ge {m\over8}\bigl( (2j+3)^2-1\bigr) = {m\over2}(j+1)(j+2), $$
which gives (\ref{nicelowerbound}).
%We observe that the 
\end{proof}

The lower bound in (\ref{lowerupperbound}) is a decreasing function of $m$
and the upper bound is an increasing function of $m$.
This says that there is more room in $\Hd$ for tight equiangular 
lines than there is in $\Cd$, and in turn $\Rd$.

\begin{example}
\label{fivetightH^3}
(Five tight equiangular lines in $\HH^3$)
By Proposition \ref{tightlinesboundofn},
there cannot be five tight equiangular lines
in $\RR^3$ or $\CC^3$, but they could exist in $\HH^3$.
We now construct such lines
as a complementary tight frame. The following five tight equiangular lines in 
$\HH^2$ with $\gl={3\over8}$ are given by \cite{B20}
$$ V= \pmat{ 1 & {\sqrt{3}\over2\sqrt{2}} & {\sqrt{3}\over2\sqrt{2}} & {\sqrt{3}\over2\sqrt{2}} & {\sqrt{3}\over2\sqrt{2}} \cr
0 & {\sqrt{5}\over2\sqrt{2}} 
& -{\sqrt{5}\over6\sqrt{2}}+{\sqrt{5}\over3}i 
& -{\sqrt{5}\over6\sqrt{2}}-{\sqrt{5}\over6}i + {\sqrt{5}\over2\sqrt{3}}j 
& -{\sqrt{5}\over6\sqrt{2}}-{\sqrt{5}\over6}i - {\sqrt{5}\over2\sqrt{3}}j 
}. $$
The complementary tight frame therefore gives five equiangular lines in $\HH^3$ at 
angle $\gl={1\over6}$. A concrete presentation of these lines is
$$ W=\pmat{ 1 & -{1\over\sqrt{6}} & -{1\over\sqrt{6}} & -{1\over\sqrt{6}} & -{1\over\sqrt{6}}  \cr
0 & {\sqrt{5}\over\sqrt{6}} & 
-{\sqrt{5}\over3\sqrt{6}}-{\sqrt{5}\over3\sqrt{3}}i & 
-{\sqrt{5}\over3\sqrt{6}}+{\sqrt{5}\over6\sqrt{3}}i-{\sqrt{5}\over6}j & 
-{\sqrt{5}\over3\sqrt{6}}+{\sqrt{5}\over6\sqrt{3}}i+{\sqrt{5}\over6}j  \cr
0 & 0 & {\sqrt{5}\over3} & 
-{\sqrt{5}\over6}+{\sqrt{5}\over2\sqrt{3}} k &
-{\sqrt{5}\over6}-{\sqrt{5}\over2\sqrt{3}} k
}. $$
This was obtained by the following general method. The condition $VV^*=AI$ for $V$ to be a tight frame 
%is that the conjugates of the rows of $V$ are orthogonal and of equal length, 
is that the entrywise conjugates of the rows of $V$ are orthogonal and of equal length, 
i.e., $V^*$ has orthogonal
columns of equal length. By using Gram-Schmidt, add orthogonal columns of equal length 
to obtain $[V^*, W^*]$ a scalar multiple of a unitary matrix. 
Then $W$ is a tight frame, which is the complement of $V$, since 
%We then have
%$$ \pmat{V^*&W^*}^* \pmat{V^*&W^*}
%= \pmat{V\cr W} \pmat{V^*&W^*}
%=\pmat{VV^*&VW^*\cr WV^*&WW^*}
%=A\pmat{I&0\cr0&I}, $$
$$ \pmat{V^*&W^*} \pmat{V^*&W^*}^* 
= \pmat{V^*&W^*} \pmat{V\cr W}  = V^*V+W^*W=AI. $$
%so that $W$ is tight frame ($WW^*=AI$) and it is the 
%complement of $V$, since $W^*W=AI-V^*V$.
\end{example}

Above we used the fact that the columns of
%Above we used the fact that the conjugates of the rows of 
the square matrix $\pmat{V^*& W^*}$ over $\HH$ are orthogonal.
For frames over $\CC$, this is equivalent to the rows
being orthogonal. For frames over the quaternions,
it is necessary to make this distinction. Indeed, there
exist unitary matrices (orthogonal columns) whose rows
are not orthogonal, e.g.,
$$ U:={1\over\sqrt{2}}\pmat{1&i\cr j&k}, \qquad
U^*U=UU^*=\pmat{1&0\cr0&1}, \qquad (U^T)^*(U^T)=\pmat{1&j\cr-j&1}. $$

\begin{example} 
\label{sixtightH^4}
By Proposition \ref{tightlinesboundofn},
there cannot be six tight equiangular lines
in $\RR^4$ or $\CC^4$, but they do exist in $\HH^4$,
by taking the complementary tight frame to the six tight equiangular
	lines in $\HH^2$ of \cite{K08}, \cite{B20} (obtained independently).
\end{example}

%The six complementary lines
%$$ \pmat{1 & -{1\over\sqrt{10}} & -{1\over\sqrt{10}} & -{1\over\sqrt{10}} & -{1\over\sqrt{10}} & -{1\over\sqrt{10}} \cr
%0 & {3\over\sqrt{10}} &-{3\over4\sqrt{10}}-{\sqrt{3}\over4\sqrt{2}}i  
%& -{3\over4\sqrt{10}}+{1\over4\sqrt{6}}i-{1\over2\sqrt{3}}j
%& -{3\over4\sqrt{10}}+{1\over4\sqrt{6}}i+{1\over4\sqrt{3}}j-{1\over4}k
%& -{3\over4\sqrt{10}}+{1\over4\sqrt{6}}i+{1\over4\sqrt{3}}j+{1\over4}k \cr
%} $$

We now consider tight equiangular lines in general,
before giving a striking summary of the known results for two dimensions.
For $n$ tight equiangular lines in $\Hd$ (or $\Cd$, $\Rd$), the angle is
$$ \gl = {n-d\over d(n-1)}, \qquad n>d, $$
with the following specific cases (in order of the number of vectors)
$$ \gl=0 \quad\hbox{(orthonormal basis)}, \qquad
 \gl={1\over d^2} \quad\hbox{(vertices of a simplex)}, $$
and sets of lines giving the bounds of Theorem \ref{linearindepth}
$$ \gl={1\over d+2}, \qquad \gl={1\over d+1} \quad\hbox{(SIC)}, \qquad
\gl={1\over d+{1\over2}}\quad \hbox{(maximal set of lines in $\Hd$}). $$
The theory as is stands does not preclude the bounds above being reached by 
lines from a larger space, e.g., $n={1\over2}d(d+1)$ complex or even 
quaternionic lines in $\Hd$. This does not occur for two dimensions.
Since 
%$$ {d\gl\over dn} = {d-1\over d(n-1)^2}>0, $$
$$ \Bigl({\partial\gl \over \partial n}\Bigr)_d  = {d-1\over d(n-1)^2}>0, $$
$\gl$ increases with the number of tight equiangular lines $n$ (for $d$ fixed),
taking the possible values
$$ \gl=0,{1\over d^2}, \ldots, {1\over d+2},\ldots, {1\over d+1}, \ldots, {2\over 2d+1}. $$

Equiangular lines are classified up to projective unitary equivalence
(see Section \ref{projectequivsect}). \\
 In two dimensions, the tight equiangular 
lines given by an orthonormal basis, the Mercedes-Benz frame and the SIC (two, three and 
four vectors, respectively) are well known, as is their uniqueness in $\CC^2$.
Putting these examples together with the five and six sets of equianglar lines
of \cite{K08}, \cite{B20} gives a complete characterisation of equiangular lines in $\HH^2$.

\begin{theorem} There is a unique set of $n$ tight equiangular
lines in $\HH^2$ for $n=2,3,4,5,6$,
with corresponding angles
$\gl=0,{1\over4},{1\over3},{3\over8},{2\over 5}$.
\end{theorem}

\subsection{Equi-isoclinic and equichordal subspaces}

We now consider generalisations of equiangularity of lines to 
$r$-subspaces ($r$-dimensional subspaces). Let $P_j$ and $P_k$ 
be the orthogonal projections onto $r$-subspaces $V_j$ and $V_k$.
Then 
$$ \norm{P_j-P_k}_F^2
= \trace((P_j-P_k)^2)
%= \trace(P_j^2+P_j^2-P_jP_k-P_kP_j)
= 2r-\trace(P_jP_k+P_kP_j) \ge 0. $$
For $\Fd=\Rd,\Cd$, we have $\trace(P_jP_k)=\trace(P_kP_j)\in\RR$, and a collection of 
$r$-subspaces is said to be {\bf equichordal} (see \cite{FJMW17}) if 
the corresponding orthogonal projections satisfy
$$ \inpro{P_j,P_k}_F = \trace(P_jP_k)=\gl r, \qquad j\ne k, $$
which reduces to the equiangularity condition (\ref{equiangulardefn})
in the case of lines ($r=1$).

For $\Hd$, $\trace(P_jP_k)$ need not be real, 
nor equal to $\trace(P_kP_j)$,  e.g., for
$$ P={1\over 2}\pmat{1&-i\cr i&1}, \quad
Q= {1\over 2}\pmat{1&-j\cr j&1}, \qquad
PQ= {1\over 4}\pmat{1-k&-i-j\cr i+j&1-k}, $$
and $\trace(PQ)={1\over2}(1-k)\ne{1\over2}(1+k)=\trace(QP)$.
%, \quad QP= {1\over 4}\pmat{1+k&-i-j\cr i+j&1+k}. 
However, by (\ref{Htraceformulareal}), we do have
$$ \trace(P_jP_k+P_kP_j) = \Re(\trace(P_jP_k+P_kP_j))
=2\Re(\trace(P_jP_k)), $$
and so we say that $r$-subspaces in $\Hd$ (or $\Rd,\Cd$) 
are {\bf equichordal} if
\begin{equation}
\label{equichordalHdefn}
\Re(\trace(P_jP_k))=\gl r, \quad j\ne k
\Iff\norm{P_j-P_k}_F^2=2(1-\gl)r, \quad j\ne k.
\end{equation}

Two $r$-subspaces $\cV_j$ and $\cV_k$, $j\ne k$, are {\bf isoclinic} with parameter
$0\le\gl\le1$ (see \cite{LS73}, \cite{H76}, \cite{FIJM24}) 
if the orthogonal projection $P_{jk}$ onto $\cV_j+\cV_k$ satisfies
$$ (1-\gl) P_{jk} = (P_j-P_k)^2. $$
%Note if $\gl=0$, then the projections are orthogonal, and we have the usual formula
%$$ P_{jk}=(P_j-P_k)^2=P_j^2-P_jP_k-P_kP_j+P_k^2=P_j+0+0+P_k=P_j+P_k. $$
%If $\gl=1$, then 
%$$ (P_j-P_k)^2 P_j =0 \Implies P_j=P_jP_kP_j \Implies V_j=V_k. $$
An equivalent condition to being isoclinic is
\begin{equation}
\label{isoclinicequivcdn}
P_jP_kP_j = \gl P_j,\qquad P_kP_jP_k = \gl P_k, \qquad j\ne k,
\end{equation}
which follows from the observation
$$ (1-\gl)P_j=(P_j-P_k)^2P_j \Iff P_jP_kP_j=\gl P_j. $$
Hoggar \cite{H76} claims that just one of the conditions (\ref{isoclinicequivcdn})
is required (over $\HH$),
which follows by writing $P_j=V_jV_j^*$, $V_j^*V_j=I$, and the implications
$$ P_jP_kP_j=\gl P_j \Iff
(V_j^*V_k)(V_k^*V_j)=\gl I
\Iff V_k^*V_j V_j^*V_k =\gl I
\Iff P_kP_jP_k=\gl P_k. $$
%which certainly is true for $\Rd$ (see \cite{LS73}) and
%$\Cd$ (since all $m$-products can be calculated, see \cite{KMW19}).
Subspaces $(V_j)$ are said to be {\bf equi-isoclinic} with parameter $0\le\gl\le1$ if
(\ref{isoclinicequivcdn}) holds.  %for all $j\ne k$. 
Equi-isoclinic subspaces are equichordal, since
$$ P_jP_kP_j=\gl P_j \Implies
\Re(\trace(P_jP_k))=\Re(\trace(P_jP_kP_j))= \trace(\gl P_j)=\gl r. $$
The orthogonal complement $(V_j^\perp)$ of equichordal subspaces is equichordal, 
since
$$ \Re(\trace((I-P_j)(I-P_k)))
= d-r-r+\Re(\trace(P_jP_k))=d-2r+\gl r, \quad j\ne k. $$
However, the orthogonal complements $(V_j^\perp)$ of equi-isoclinic subspaces
$(V_j)$ are
not in general equi-isoclinic, as the following example shows. 

\begin{example} (Two isoclinic planes do not exist in $\RR^3$).
Consider the equi-isoclinic $1$-dimensional subspaces given by 
%the unit vectors $v_1$ and $v_2$ in $\RR^3$ 
$$ v_1=\pmat{1\cr0\cr0}, \quad
v_2=\pmat{\sqrt{1-a^2-b^2}\cr a\cr b}. $$
%Q_1=I-v_1v_1^*=\pmat{1&0&0 \cr 0&1&0\cr 0&0&1}
The orthogonal projections $Q_j=I-v_jv_j^*$ onto the complementary subspaces satisfy
$$ Q_1Q_2Q_1=\pmat{0&0&0\cr 0& 1-a^2 & -ab\cr 0& -ab & 1-b^2},
\qquad Q_1=\pmat{0&0&0\cr0&1&0\cr0&0&1}. $$
Hence, for $V_1^\perp$ and $V_2^\perp$ to be isoclinic, we must have that $a=b=0$,
i.e., $V_1=V_2$. Thus there cannot be two (nonequal) isoclinic planes in $\RR^3$,
despite the fact that there can be up to six equi-isoclinic lines in $\RR^3$.
\end{example}

%It also appears from the definition that the complements of isoclinic subspaces are
%not isoclinic subspaces. Are there two isoclinic planes in $\RR^3$?
%
%\begin{align*}
%(I-P_j)(I-P_k)(I-P_j)
%& = I-P_j-P_k+P_jP_k -P_j+P_j^2+P_kP_j-P_jP_kP_j \cr
%& = I-P_j-P_k+P_jP_k +P_kP_j-P_jP_kP_j \cr
%& = I-(P_j-P_k)^2-P_jP_kP_j \cr
%& = I-(P_j-P_k)^2-\gl P_j 
%\end{align*}
%applying the RHS to $x=P_jx$ gives
%$$ x-(P_j-P_k)^2P_j x-\gl P_j x
%=  x-(1-\gl)x-\gl x= 0, $$
%and applying it to $x\in V_j^\perp$ gives
%$$ x-P_k x+P_jP_k x
%= x-(I-P_j)P_k x
%= (I-P_j)(I-P_k) x $$

%We note the Frobenius distance
%$$ \norm{P_j-P_k}_F^2 = \trace((P_j-P_k)^2)
%=\trace(P_j)-2\trace(P_jP_k)+\trace(P_k)\ge 0, $$
%so if $\dim(V_j)=\dim(V_k)=r$, we have
%$$ 0\le \trace(P_jP_k) =\gl r\le r. $$

\section{From $\RR$ to $\CC$ and $\CC$ to $\HH$, and back}
\label{RtoCtoH}

There is a natural inclusion $\RR\subset\CC\subset\HH$ and
hence of $\Rd\subset\Cd\subset\Hd$. Since tight frames are determined up to 
unitary equivalence by their Gramians:
\begin{itemize}
\item There is a unitary map of a tight frame to $\Rd$ if and only if its Gramian has real entries,
and we say the tight frame is {\bf real}.
\item There is a unitary map of a tight frame to $\Cd$ if and only if its Gramian has complex entries,
and we say the tight frame is {\bf complex} if its Gramian has a nonreal entry.
\item If the Gramian of a tight frame has a noncomplex entry, 
then we say that it is a {\bf quaternionic} tight frame.
\end{itemize}
As an example, the four equiangular lines in $\HH^2$ of Hoggar \cite{H76} 
are lines in $\CC^2$ (see Example \ref{HoggarlinesinC^2}).
For tight frames up to projective unitary equivalence, i.e., thought of as lines, the 
corresponding analogue is more involved, see Section \ref{projectequivsect}.

There is also a natural identification of a point $z=x+iy\in\CC$ 
(in the complex plane) with a point $(x,y)\in\RR^2$ (in the plane). 
We generalise this, by defining an 
invertible $\RR$-linear map
\begin{equation}
\label{CdR2dcorrespondence}
[\cdot]_\RR: \Cd\to\RR^{2d}: v\mapsto\pmat{\Re v\cr\Im v}, \qquad
\Re v={v+\overline{v}\over2}, \quad \Im v={v-\overline{v}\over 2i}.
\end{equation}
Based on a thorough analysis of this, we will then define an analogous map $\Hd\to\CC^{2d}$.
The first subtle point, is that
$[\cdot]_\RR$ maps $k$-dimensional complex-subspaces of $\Cd$ to real $(2k)$-dimensional subspaces of $\RR^{2d}$.
To see why this is, 
we first calculate the image of a 
complex scalar multiple $\ga+i\gb$ of a vector $v=x+iy$ 
$$ (\ga+i\gb) v 
= (\ga+i\gb) (x+iy) 
= \ga x-\gb y+i(\ga y+\gb x), $$
%$$ v (\ga+i\gb)
%= (\Re(v)+i\Im(v)) (\ga+i\gb)
%=\ga\Re(v)-\gb\Im(v)+i(\ga\Im(v)+\gb\Re(v)), $$
%$$ \Re((\ga+i\gb)v)
%={ (\ga+i\gb)v +(\ga-i\gb)\overline{v}\over 2}
%=\ga{v+\overline{v}\over2}-\gb{v-\overline{v}\over 2i}
%=\ga\Re(v)-\gb\Im(v), $$
%$$ \Im((\ga+i\gb)v)
%={ (\ga+i\gb)v -(\ga-i\gb)\overline{v}\over 2i}
%=\gb{v+\overline{v}\over2}+\ga{v-\overline{v}\over 2i}
%=\ga\Im(v)+\gb\Re(v), $$
which gives
\begin{equation}
\label{Csubspaceimage}
[ (\ga+i\gb) v ]_\RR 
= \ga \pmat{\Re v\cr\Im v} 
+ \gb \pmat{-\Im v\cr\Re v} 
= \ga [v]_\RR + \gb [iv]_\RR.
\end{equation}
%$$ (\ga+i\gb)v\mapsto (\ga\Re(v)-\gb\Im(v), \ga\Im(v)+\gb\Re(v) ) 
%=\ga(\Re(v),\Im(v)) +\gb(-\Im(v),\Re(v) ). $$
Thus the one-dimensional complex subspace spanned by $v\in\Cd$
is mapped to the real two-dimensional subspace
$$ [\spam_\CC\{v\}]_\RR= \spam_\RR\{\pmat{\Re v\cr\Im v},\pmat{-\Im v\cr\Re v}\} \qquad
\hbox{(orthogonal vectors in $\RR^{2d}$)}. $$
The general result
% for $k$-dimensional subspaces 
then follows from the correspondence between linear
dependencies
%and there is a the natural correspondence between linear dependencies
$$ \sum_\ell (\ga_\ell+i\gb_\ell)v_\ell =0 \Iff
 \sum_\ell\Bigl\{ \ga_\ell\pmat{\Re v_\ell\cr\Im v_\ell }+\gb_\ell\pmat{-\Im v_\ell \cr\Re v_\ell}
\Bigr\} =0. $$
%and a dimension count.
We also calculate
%\begin{align*} 
%** \inpro{v,w} 
%&= \inpro{\Re v+i\Im v,\Re w+i\Im w} \cr
%%= (\Re v+i\Im v)\cdot(\Re w-i\Im w) \cr
%&= \inpro{\Re v,\Re w}+\inpro{\Im v,\Im w} 
%+ i (\inpro{\Im v,\Re w}-\inpro{\Re v,\Im w}),
%\end{align*}
\begin{align*}
\inpro{v,w}
&= \inpro{\Re v+i\Im v,\Re w+i\Im w} \cr
%= (\Re v+i\Im v)\cdot(\Re w-i\Im w) \cr
&= \inpro{\Re v,\Re w}+\inpro{\Im v,\Im w}
+ i ( \inpro{\Re v,\Im w} - \inpro{\Im v,\Re w}),
\end{align*}
%\begin{align*} 
%\inpro{v,w} 
%&= \inpro{\Re(v)+i\Im(v),\Re(w)+i\Im(w)}
%= (\Re(v)+i\Im(v))\cdot(\Re(w)-i\Im(w)) \cr
%&= \inpro{\Re(v),\Re(w)}+\inpro{\Im(v),\Im(w)} 
%+ i (\inpro{\Im(v),\Re(w)}-\inpro{\Re(v),\Im(w)}),
%\end{align*}
so that
\begin{equation}
\label{CtoRinpro}
\Re(\inpro{v,w}_\CC) = \inpro{[v]_\RR,[w]_\RR}_\RR, \quad
\Im(\inpro{v,w}_\CC) = \inpro{[iv]_\RR,[w]_\RR}_\RR, \qquad
\end{equation}
%and 
\begin{equation}
\label{CtoRortho}
\inpro{[v]_\RR,[iv]_\RR}_\RR=0. 
\end{equation}
%\begin{equation}
%\label{CtoRinpro}
%\Re(\inpro{v,w}_\CC) = \inpro{[v]_\RR,[w]_\RR}_\RR, \qquad
%\inpro{[v]_\RR,[iv]_\RR}_\RR=0.
%\end{equation}
%aaaaaaa
%weHere, for emphasis, we write $\inpro{\cdot,\cdot}_\FF$
Let $A:\CC^n\to\CC^m$ a $\CC$-linear map be represented as an $\RR$-linear map 
$[A]_\RR:\RR^{2n}\to\RR^{2m}$
under this identification, 
i.e., $[A]_\RR:= [\cdot]_\RR A[\cdot]_\RR^{-1}$. Then
%$$ [A]\pmat{u\cr v} 
%= \pmat{\Re(A(u+iv))\cr\Im(A(u+iv))}
%= \pmat{ {Au+iAv-(\overline{A}u-i\overline{A}v\over 2}\cr\Im(A(u+iv))} $$
\begin{align*} 
A(u+iv) &=(\Re(A)+i\Im(A))(u+iv) \cr
&= \Re(A)u-\Im(A)v+i\Im(A)u+i\Re(A)v, \quad u,v\in\RR^n,
\end{align*}
and $\Re(A^*)=\Re(A)^T$,
$\Im(A^*)=-\Im(A)^T$, 
so that
\begin{align*}
[A]_\RR &=\pmat{\Re(A)&-\Im(A)\cr\Im(A)&\Re(A)}, \qquad
\rank([A]_\RR)=2\rank(A), \cr
%$$ \Re(A(u+iv)) = { Au+iAv + (\overline{A}u-i\overline{A}v) \over 2}
%= {A+\overline{A}\over 2}u + {iA-i\overline{A}\over2}v. $$
%$$ \Im(A(u+iv)) = { Au+iAv - (\overline{A}u-i\overline{A}v) \over 2i}
%= {i\overline{A}-iA\over 2}u + {A+\overline{A}\over2}v. $$
%$$ [A] = \pmat{ {A+\overline{A}\over2} & {iA-i\overline{A}\over2} \cr
%{i\overline{A}-iA\over 2} & {A+\overline{A}\over2} }
%=\pmat{\Re(A)&-\Im(A)\cr\Im(A)&\Re(A)}. $$
%$$ \Re(A^*)={A^*+\overline{A^*}\over2}
%= ({A+\overline{A}\over2})^*= \Re(A)^T, $$
%$$ \Im(A^*)={A^*-\overline{A^*}\over2i} = ({A-\overline{A}\over -2 i})^*
%= -\Im(A)^T. $$
[A^*]_\RR&=\pmat{\Re(A)^T&\Im(A)^T\cr-\Im(A)^T&\Re(A)^T} = [A]_\RR^T.
\end{align*}
%$$ \trace([A])=\trace(A)+\trace(\overline{A})=2\trace(\Re(A)). $$
The usual rules for matrix multiplication follow, 
e.g, $[A]_\RR [B]_\RR = [AB]_\RR$. 
%$$ [A][B]
%= \pmat{\Re(A)\Re(B) -\Im(A)\Im(B)&-\Re(A)\Im(B)-\Im(A)\Re(B) \cr
%\Re(A)\Im(B)+\Im(A)\Re(B)& \Re(A)\Re(B) -\Im(A)\Im(B)} =[AB]. $$
%\begin{align*}
%AB &= (\Re(A)+i\Im(A))(\Re(B)+i\Im(B)) \cr
%&= (\Re(A)\Re(B) -\Im(A)\Im(B)) +i(\Re(A)\Im(B)+\Im(A)\Re(B)).
%\end{align*}
One must be careful if a vector $v\in\Cd$ is being thought of as
a $d\times 1$ matrix, i.e., the linear map $[v]:\CC\to\Cd:\ga\mapsto\ga v$,
since $[v]_\RR\in\RR^{2d\times 1}$, $[[v]]_\RR\in\RR^{2d\times 2}$.
In particular, the familiar formula $P=vv^*$ for the orthogonal projection
onto a unit vector $v\in\Cd$, is $P=[v][v]^*$, which maps as follows
$$ [P]_\RR=[[v]]_\RR[[v]^*]_\RR=[[v]]_\RR[[v]]_\RR^T, \qquad
[[v]]_\RR=\pmat{\Re v&-\Im v\cr\Im v&\Re v} . $$
This is the orthogonal projection onto
$$ [\spam_\CC\{v\}]_\RR
=\spam_\RR\{[v]_\RR,[iv]_\RR\}, \qquad
[v]_\RR=\pmat{\Re v\cr\Im v}, \quad [iv]_\RR=\pmat{-\Im v\cr\Re v}. $$
The identification $[\cdot]_\RR$ preserves various properties of linear maps, 
see Theorem \ref{mapspropspreserved}.
In particular, orthogonal projections map to orthogonal projections, and
hence:
\begin{itemize}
\item Equi-isoclinic subspaces of dimension $r$ in $\Cd$ 
correspond to equi-isoclinic subspaces of dimension $2r$ in $\RR^{2d}$,
and similarly for equichordal subspaces.
\end{itemize}

We now consider the situation for tight frames, which is somewhat more involved, e.g., 
a basis for $\Cd$ does not correspond to a basis for $\RR^{2d}$
(which has twice the dimension).
%since $P_jP_kP_j=\gl P_j$ gives
%$$ [P_j]_\RR[P_k]_\RR[P_j]_\RR=\gl[P_j]_\RR. $$
%
%Every sequence of nonzero vectors $\CC$ is a tight frame, but (obviously) not
%all correspond to tight frames in $\RR^2$. 
%
Let $V=V_1+iV_2$ be the synthesis map for a sequence of vectors $v_1,\ldots,v_n\in\Cd$, 
and $V_\RR$ be the corresponding map for the sequence $[v_1]_\RR,\ldots,[v_n]_\RR\in\RR^{2d}$,  i.e.,
$$ V_\RR= \pmat{V_1\cr V_2}\in\RR^{2d\times n}. $$
Then $V$ gives a tight frame for $\Cd$ if and only if
$$ VV^*=(V_1+iV_2)(V_1^*-iV_2^*)=V_1V_1^*+V_2V_2^*+i(V_2V_1^*-V_1V_2^*)=AI, $$
where 
$dA := \sum_j\norm{v_j}^2 = \trace(VV^*)=\trace(V_\RR V_\RR^T)= $,
i.e.,
$$ V_1V_1^T+V_2V_2^T=AI, \qquad V_2V_1^T-V_1V_2^T=0, $$
and $V_\RR$ gives a tight frame for $\RR^{2d}$ if and only if
$$ V_\RR V_\RR^*=\pmat{V_1\cr V_2}\pmat{V_1^T & V_2^T}
=\pmat{V_1V_1^T& V_1V_2^T\cr V_2V_1^T& V_2V_2^T}
= {1\over 2} A \pmat{ I&0\cr0& I}, $$
i.e.,
\begin{equation}
\label{Sversionofequiv}
V_1V_1^T = V_2 V_2^T = {1\over 2}AI,  \qquad V_1V_2^T=V_2V_1^T=0.
\end{equation}
Thus all tight frames for $\RR^{2d}$ map to tight frames for $\Cd$, and 
a tight frame for $\Cd$ gives a tight frame for $\RR^{2d}$ if and only if
(\ref{Sversionofequiv}) holds.
This condition says that $V_1$ and $V_2$ are tight frames for $\RR^d$
(with the same frame bound) which are orthogonal (see \cite{W18} \S 3.5).
We now show that 
(\ref{Sversionofequiv}) % it only 
depends only on $V$ up to unitary equivalence.

Let $U=U_1+iU_2$ be unitary, then
$UU^*=U_1U_1^T+U_2U_2^T+i(U_2U_1^T-U_1U_2^T)=I$, which is equivalent to
\begin{equation}
\label{unitarycdnI}
U_1U_1^T+U_2 U_2^T = I, \qquad U_2U_1^T-U_1U_2^T=0.
\end{equation}
Suppose that $V$ satisfies (\ref{Sversionofequiv}), then 
$$ UV=[Uv_1,\ldots,Uv_n]= (U_1+iU_2)(V_1+iV_2)=(U_1V_1-U_2V_2)+i(U_2V_1+U_1V_2), $$
$A=\sum_j\norm{v_j}^2=\sum_j \norm{Uv_j}^2$,
and using (\ref{unitarycdnI}), we calculate
\begin{align*}
\Re(UV)\Re(UV)^T &= (U_1V_1-U_2V_2) (V_1^TU_1^T-V_2^TU_2^T) 
%&= U_1(V_1V_1^T)U_1^T - U_1(V_1V_2^T)U_2^T -U_2(V_2V_1^T)U_1^T +U_2(V_2V_2^T)U_2^T\cr
= {1\over2} A (U_1U_1^T + U_2U_2^T)
= {1\over2}A I, \cr
\Im(UV)\Im(UV)^T 
&= (U_2V_1+U_1V_2)(V_1^TU_2^T+V_2^TU_1^T)
= {1\over2}A(U_2U_2^T+U_1U_1^T)= {1\over2}AI, \cr
\Re(UV)\Im(UV)^T &= (U_1V_1-U_2V_2) (V_1^TU_2^T+V_2^TU_1^T)
%= U_1V_1V_1^TU_2^T-U_2V_2V_2^TU_1
= {1\over2} A (U_1U_2^T-U_2U_1^T)= 0, 
\end{align*}
so that $UV$ satisfies (\ref{Sversionofequiv}).

Since the condition for a tight frame for $\Cd$ to be a tight frame for $\RR^{2d}$ depends
only on $V$ up to unitary equivalence, it follows %(see \cite{W20}) 
that this condition 
can be written in terms of the Gramian of $V$. The Gramians of $V$ and $V_\RR$ are
%\begin{align*}
%V^*V &= (V_1^*-iV_2^*) (V_1+iV_2) = V_1^TV_1+V_2^TV_2+i(V_1^TV_2-V_2^TV_1), \cr
%V_\RR^*V_\RR &= \pmat{V_1^T& V_2^T} \pmat{V_1\cr V_2} = V_1^TV_1+V_2^TV_2.
%\end{align*}
$$ V^*V= (V_1^*-iV_2^*) (V_1+iV_2) = V_1^TV_1+V_2^TV_2+i(V_1^TV_2-V_2^TV_1), $$
$$ V_\RR^*V_\RR = \pmat{V_1^T& V_2^T} \pmat{V_1\cr V_2} = V_1^TV_1+V_2^TV_2. $$
The variational characterisation for being a tight frame for $\Cd$ and for $\RR^{2d}$
are %(respectively)
$$ \norm{V^*V}_F^2 = {1\over d}(\trace(V^*V))^2, \qquad
\norm{V_\RR^*V_\RR}_F^2 = {1\over 2d}(\trace(V_\RR^*V_\RR))^2. $$
Since $\trace(V^*V)=  \trace(V_1^TV_1+V_2^TV_2) =\trace(V_\RR^* V_\RR)$, 
a tight frame for $\Cd$ gives a tight frame for $\RR^{2d}$ if and only if
\begin{equation}
\label{CtoRcdnI}
2 \norm{V_\RR^*V_\RR}_F^2 -\norm{V^*V}_F^2  =0. 
\end{equation}
By writing this explicitly in terms of $V^*V$ %  the Gramian of $V$ 
(cf \cite{W20b}),
we obtain the following.

%We have
%$$ \trace(V^*V)=  \trace(V_1^TV_1+V_2^TV_2) =\trace(V_\RR^* V_\RR), $$
%$$ \norm{V_\RR^* V_\RR}_F^2 = \trace\bigr( (V_1^TV_1+V_2^TV_2)^2\bigl), $$
%\begin{align*}
%\norm{V^*V}_F^2
%&=\trace( (V^*V)^2)
%= \trace\bigl( (V_1^TV_1+V_2^TV_2)^2
%+(V_1^TV_2-V_2^TV_1)(V_1^TV_2-V_2^TV_1)^T 
%\bigr) \cr
%& = \norm{V_\RR^*V_\RR}_F^2+\norm{V_1^TV_2-V_2^TV_1}_F^2.
%\end{align*}
%Thus if $V_\RR$ is tight frame, the condition for $V$ to be is
%$$ \norm{V_\RR^*V_\RR}_F^2+\norm{V_1^TV_2-V_2^TV_1}_F^2
%= 2\norm{V_\RR^*V_\RR}_F^2. 
%$$

\begin{theorem} 
\label{tightframesRtoC}
Let $[\cdot]_\RR:\Cd\to\RR^{2d}$ be the correspondence (\ref{CdR2dcorrespondence}) between
$\Cd$ and $\RR^{2d}$. Then 
\begin{enumerate}
\item Tight frames for $\RR^{2d}$ correspond to tight frames for $\Cd$.
\item A tight frame $V=[v_1,\ldots,v_n]$ for $\Cd$ corresponds to a tight frame
for $\RR^{2d}$ if and only if it satisfies
\begin{equation}
\label{CtoRcdnII}
\sum_j\sum_k \inpro{v_j,v_k}^2 =0,
\end{equation}
which can also be written as
\begin{equation}
\label{CtoRcdnIIextra}
\sum_j\sum_k (\Re \inpro{v_j,v_k})^2 
= \sum_j\sum_k (\Im \inpro{v_j,v_k})^2.
\end{equation}
%which is equivalent to
%\begin{equation}
%\label{RtoCframeopcdn}
%V_1V_1^T = V_2 V_2^T = {1\over 2}AI,  \qquad V_1V_2^T=V_2V_1^T=0,
%\qquad dA:=\sum_j\norm{v_j}^2.
%\end{equation}
%where $V=V_1+iV_2$, with $V_1,V_2\in\RR^{d\times n}$.
%, and $dA:=\sum_j\norm{v_j}^2$. 
%This condition depends only on $V$ up to unitary
%equivalence, and is equivalent to following condition on Gramian
\end{enumerate}
\end{theorem}

\begin{proof} In light of our previous discussion, it remains only to
show that (\ref{CtoRcdnI}) can be written as (\ref{CtoRcdnII}) and
(\ref{CtoRcdnIIextra}).
Using (\ref{CtoRinpro}), we have
\begin{align*}
\norm{V_\RR^*V_\RR}_F^2 -\norm{V^*V}_F^2
 &= 2\sum_j\sum_k\inpro{[v_j]_\RR,[v_k]_\RR}^2 - \sum_j\sum_k|\inpro{v_j,v_k}|^2 \cr
& = 2\sum_j\sum_k(\Re\inpro{v_j,v_k})^2 
- \sum_j\sum_k|\inpro{v_j,v_k}|^2 =0.
%- \sum_j\sum_k\Bigl((\Re\inpro{v_j,v_k})^2+(\Im\inpro{v_j,v_k})^2\Bigr)=0,
\end{align*}
By taking $z=\inpro{v_j,v_k}$ in 
$$ 2(\Re(z))^2 - |z|^2
= 2\Bigl({z+\overline{z}\over2}\Bigr)^2 - z\overline{z}
= {1\over2} (z^2+\overline{z}^2), $$
we see that this condition can be written as
$$ {1\over2}\sum_j\sum_k \bigl( \inpro{v_j,v_k}^2+\inpro{v_k,v_j}^2\bigr)
=\sum_j\sum_k\inpro{v_j,v_k}^2=0. $$
which gives (\ref{CtoRcdnII}).
By substituting in 
%$$ |\inpro{v_j,v_k}|^2=(\Re\inpro{v_j,v_k})^2+(\Im\inpro{v_j,v_k})^2, $$
$ |\inpro{v_j,v_k}|^2=(\Re\inpro{v_j,v_k})^2+(\Im\inpro{v_j,v_k})^2, $
we obtain (\ref{CtoRcdnIIextra}).
\end{proof}

\begin{example} A tight frame $(z_j)$ for $\CC$ corresponds to a tight frame for $\RR^2$
if and only if 
$$ 
\sum_j\sum_k (z_j\overline{z_k})^2
= \Bigl(\sum_j z_j^2\Bigr)\Bigl(\sum_k \overline{z_k}^2 \Bigr)
= \Bigl|\sum_j z_j^2\Bigr|^2=0 \Iff
\sum_j z_j^2=0. $$
The complex number $z_j^2=(x_j+iy_j)^2$ corresponding to a point $(x_j,y_j)$ is 
sometimes called a diagram vector, and the condition that a frame for $\RR^2$ is tight if and
only if its diagram vectors sum to zero is well known.
\end{example}

%Let $\HH$ be the quaternions, with $\overline{q}$ the conjugate. 
We now give a map $\HH^d\to\CC^{2d}$
that has similar properties to $[\cdot]_\RR:\Cd\to\RR^{2d}$. 
This is based on the following analogue of the polar decomposition
for $\CC$, the Cayley-Dickson construction, that every quaternion $q\in\HH$ can be written uniquely
\begin{equation}
\label{jswap}
q=z+w j, \qquad z,w\in\CC.
\end{equation}
Moreover, we observe the ``commutativity'' relation
$$ jz = \overline{z}j, \qquad\forall z\in\CC, $$
which implies
$$ jA=\overline{A}j, \qquad\forall A\in\CC^{m\times n}. $$
Let $\HH^d$ be a right vector space,
and define a $\CC$-linear map
\begin{equation}
\label{HdC2dcorrespondence}
[\cdot]_\CC:\HH^d\to\CC^{2d}:z+wj\mapsto\pmat{z\cr\overline{w}}, 
\end{equation}
The conjugation $\overline{w}$ is necessary for $\CC$-linearity:
$(z+wj)\ga=z\ga+wj\ga=z\ga+w\overline{\ga}j$ gives
$$ [(z+wj)\ga]_\CC 
%=\pmat{z\ga\cr\overline{w\overline{\ga}}}
=\pmat{z\ga\cr\overline{w}\ga} 
=\pmat{z\cr\overline{w}}\ga 
=[z+wj]_\CC\ga 
\qquad\forall\ga\in\CC. $$
Let $\Co_1$ and $\Co_2$ be the $\CC$-linear maps $\Hd\to\Cd$ 
giving the ``complex coordinates'' of $q=z+wj$, i.e.,
$$ \Co_1(z+wj):=z, \qquad \Co_2(z+wj):=\overline{w}.$$
We note in particular, that
$$ |q|^2=|\Co_1(q)|^2+|\Co_2(q)|^2. $$
From 
\vskip-1.0truecm
%\begin{align*}
%\inpro{v,w}_\HH
%&= \inpro{v_1+v_2j,w_1+w_2j} \cr
%%= \inpro{v_1,w_1}+\inpro{v_1,w_2j} +\inpro{v_2j,w_1}+\inpro{v_2j,w_2j} \cr
%%&= \inpro{v_1,w_1}-j\inpro{v_1,w_2}+\inpro{v_2,w_1}j-j\inpro{v_2,w_2}j 
%%= \inpro{v_1,w_1}-\inpro{ w_2 , v_1 }j+\inpro{v_2,w_1}j+\inpro{w_2,v_2} \cr
%%&= (\inpro{v_1,w_1}+\inpro{\overline{v_2},\overline{w_2}})
%%+(\inpro{v_2,w_1} -\inpro{ w_2 , v_1 })j, \cr
%	&= (\inpro{v_1,w_1}+\inpro{\overline{v_2},\overline{w_2}})
%	-( \inpro{-v_2,w_1} +\inpro{\overline{v_1},\overline{w_2}} )j, \cr
%	& = \inpro{[v]_\CC,[w]_\CC}_\CC-\inpro{[vj]_\CC,[w]_\CC}_\CC j
%\end{align*}
\begin{align*}
\inpro{v,w}_\HH
&= \inpro{v_1+v_2j,w_1+w_2j} \cr
%= \inpro{v_1,w_1}+\inpro{v_1,w_2j} +\inpro{v_2j,w_1}+\inpro{v_2j,w_2j} \cr
%&= \inpro{v_1,w_1}-j\inpro{v_1,w_2}+\inpro{v_2,w_1}j-j\inpro{v_2,w_2}j
%= \inpro{v_1,w_1}-\inpro{ w_2 , v_1 }j+\inpro{v_2,w_1}j+\inpro{w_2,v_2} \cr
%&= (\inpro{v_1,w_1}+\inpro{\overline{v_2},\overline{w_2}})
%+(\inpro{v_2,w_1} -\inpro{ w_2 , v_1 })j, \cr
&= \inpro{v_1,w_1}-j\inpro{v_2,w_2}j -j\inpro{v_2,w_1} +\inpro{v_1,w_2}j, \cr
	&= \inpro{v_1,w_1}+\inpro{\overline{v_2},\overline{w_2}} 
	- (\inpro{\overline{v_2},\overline{w_1}} +\inpro{v_1,-w_2} )j, \cr
        & = \inpro{[v]_\CC,[w]_\CC}_\CC-\inpro{[v]_\CC,[wj]_\CC}_\CC j
\end{align*}
we get the analogues of 
(\ref{CtoRinpro})
and (\ref{CtoRortho})
\begin{equation}
\label{HtoCinpro}
\Co_1(\inpro{v,w}_\HH)=\inpro{[v]_\CC,[w]_\CC}_\CC, \qquad
\Co_2(\inpro{v,w}_\HH)=-\inpro{[v]_\CC,[wj]_\CC}_\CC.
\end{equation}
\begin{equation}
\label{HtoCortho}
\inpro{[v]_\CC,[vj]_\CC}_\CC=0. 
\end{equation}
%\begin{equation}
%\label{HtoRinpro}
%\Re(\inpro{v,w}_\HH) 
%= \Re(\Co_1(\inpro{v,w}_\HH))
%= \Re(\inpro{[v]_\CC,[w]_\CC}_\CC)
%= \inpro{[[v]_\CC]_\RR,[[w]_\CC]_\RR}_\RR.
%\end{equation}

The analogue of (\ref{Csubspaceimage}) for $v=z+wj$ is
\begin{equation}
\label{Hsubspaceimage}
[ v(\ga+\gb j) ]_\CC 
= [ v\ga+ v\gb j ]_\CC 
= [ v\ga+ v j\overline{\gb}]_\CC 
= [ v]_\CC\ga+ [v j ]_\CC \overline{\gb}, \qquad\ga,\gb\in\CC,
\end{equation}
where
$$ [v]_\CC = \pmat{z\cr\overline{w}}, \quad
[vj]_\CC = \pmat{-w\cr\overline{z}}, \qquad
\inpro{[v]_\CC,[vj]_\CC}_\CC=0.  $$
Thus $[\cdot]_\CC$ maps $k$-dimensional $\HH$-subspaces of $\Hd$ to $(2k)$-dimensional
$\CC$-subspaces of $\CC^{2d}$.

Let $L:\HH^n\to\HH^m$ be an $\HH$-linear map be represented as a $\CC$-linear map 
$[L]_\CC:\CC^{2n}\to\CC^{2m}$
under this identification, 
i.e., $[L]_\CC:= [\cdot]_\CC L[\cdot]_\CC^{-1}$. 
In view of (\ref{jswap}),
its standard matrix $[L]_\HH\in\HH^{m\times n}$
has a unique decomposition
$$ [L]_\HH=A+Bj, \qquad A,B\in\CC^{m\times n}. $$
We have
\begin{align*}
L(z+wj) & = (A+Bj)(z+wj)
= Az+Awj+Bjz+Bjwj  \cr
&= Az+Awj+B\overline{z}j-B\overline{w}
%&= Az+j\overline{A}w+jBz-\overline{B}w
= Az -B\overline{w} +\overline{( \overline{B}z + \overline{A}\overline{w})}j, 
%&= Az -\overline{B}w +j\bigl( Bz+ \overline{A}w\bigr),
\end{align*}
and %The matrix representing the adjoint $L^*$ is
$$ [L^*]_\HH = (A+Bj)^* = A^*+(-j)B^* = A^*-\overline{B^*}j=A^*-B^T j, $$
so that
\begin{align*}
%\label{HtoCdecomp}
[L]_\CC & = \pmat{ A&-{B}\cr \overline{B}&\overline{A}}, \qquad
\rank([L]_\CC)=2\rank([L]_\HH), \cr
[L^*]_\CC &= \pmat{A^*&B^T\cr -B^*&A^T}=[L]_\CC^*.
\end{align*}
The other observations for the previous case also hold 
(see Theorem \ref{mapspropspreserved}),
in particular
\begin{itemize}
\item Equi-isoclinic subspaces of dimension $r$ in $\Hd$ 
correspond to equi-isoclinic subspaces of dimension $2r$ in $\CC^{2d}$,
and similarly for equichordal subspaces.
\end{itemize}

We now seek the analogue of Theorem \ref{tightframesRtoC}, this time starting
with the development in terms of the Gramian.
The variational characterisation
%(Theorem \ref{generalisedWelchbound})
 for $V=[v_1,\ldots,v_n]$ being a tight frame for $\Hd$ and
for $V_\CC:=\bigl[[v_1]_\CC,\ldots,[v_n]_\CC\bigr]$ being a tight frame for $\CC^{2d}$
are %(respectively)
$$ \norm{V^*V}_F^2 = {1\over d}(\trace(V^*V))^2, \qquad
\norm{V_\CC^*V_\CC}_F^2 = {1\over 2d}(\trace(V_\CC^*V_\CC))^2. $$
Since $\trace(V^*V)=\trace(V_\CC^* V_\CC)$, 
a tight frame for $\Hd$ gives a tight frame for $\CC^{2d}$ if and only if
\begin{equation}
\label{HtoCcdnI}
2 \norm{V_\CC^*V_\CC}_F^2 -\norm{V^*V}_F^2  =0. 
\end{equation}
Writing this explicitly in terms of the Gramian $V^*V$ %  the Gramian of $V$ 
gives the following.

%kkkkkkk

\begin{lemma}
\label{tightframesCtoHlemma}
Let $V=[v_1,\ldots,v_n]=V_1+V_2j\in\HH^{d\times n}$. 
Then the following are equivalent
\begin{enumerate}[\rm (i)]
\item $V_\CC=\bigl[[v_1]_\CC,\ldots,[v_n]_\CC\bigr]=\pmat{V_1\cr \overline{V_2}}\in\CC^{2d\times n}$ is a tight frame for $\CC^{2d}$.
%\item $\bigl[[v_1j]_\CC,\ldots,[v_nj]_\CC\bigr] = \pmat{-V_2\cr \overline{V_1}}$ is a tight frame for $\CC^{2d}$.
\item $$ V_1V_1^*=V_2V_2^*={1\over2}AI, \qquad V_1V_2^T=V_2V_1^T=0, \qquad
A:={1\over d} \sum_j\norm{v_j}^2. $$
\item $$ 
(V_1^*V_1+V_2^T\overline{V_2})^2
= {1\over 2}A\,(V_1^*V_1+V_2^T\overline{V_2}) , \qquad
A:={1\over d}\sum_j\norm{v_j}^2. $$
\item $$ \norm{\Co_1(V^*V)}_F^2=\sum_j\sum_k |\Co_1(\inpro{v_j,v_k})|^2
%=\sum_j\sum_k |\Co_2(\inpro{v_j,v_k})|^2
={1\over 2d} \Bigl(\sum_j\norm{v_j}^2\Bigr)^2. $$
\end{enumerate}
\end{lemma}

\begin{proof} 
In terms of the frame operator, 
%the conditions (i) and (ii) are
the condition (i) is
$$ \pmat{V_1\cr \overline{V_2}} \pmat{V_1\cr \overline{V_2}}^*
= \pmat{V_1\cr \overline{V_2}} \pmat{V_1^* & V_2^T}^*
= \pmat{V_1V_1^* & V_1V_2^T\cr\overline{V_2}V_1^*&\overline{V_2}V_2^T}
= {1\over2}A\pmat{I&0\cr0&I}, $$
%$$ \pmat{-V_2\cr \overline{V_1}} \pmat{-V_2\cr \overline{V_1}}^*
%= \pmat{-V_2\cr \overline{V_1}} \pmat{-V_2^* & V_1^T}^*
%= \pmat{V_2V_2^* & -V_2V_1^T\cr-\overline{V_1}V_2^*&\overline{V_1}V_1^T}
%= A\pmat{I&0\cr0&I}, $$
where $dA=\sum_j\norm{v_j}^2$, which is clearly equivalent to (ii).

In terms of the Gramian $V_\CC^*V_\CC =V_1^* V_1+V_2^T\overline{V_2}$ 
being (a multiple of) an orthogonal projection matrix 
(Proposition %\ref{PGramianchar}
\ref{tightframeequidefs}), 
the condition (i) is (iii).

In terms of the variational characterisation 
(Theorem \ref{generalisedWelchbound}), the condition (i) is
$$ \sum_j \sum_k |\inpro{[v_j]_\CC,[v_k]_\CC}|^2 = {1\over 2d}\sum_j\Bigl(\norm{[v_j]_\CC}^2\Bigr)^2, $$
which can be written as (iv), since $\inpro{[v]_\CC,[w]_\CC}_\CC=\Co_1(\inpro{v,w}_\HH)$
and $\norm{[v]_\CC}=\norm{v}_\HH$.
% These are both clearly equivalent to (iii).
\end{proof}

We observe that condition the (iv) depends only on $V$ up to unitary equivalence,
and so the others do also.

\begin{theorem} 
\label{tightframesCtoH}
Let $[\cdot]_\CC:\Hd\to\CC^{2d}$ be the correspondence (\ref{HdC2dcorrespondence}) between
$\Hd$ and $\CC^{2d}$. Then 
\begin{enumerate}
\item Tight frames for $\CC^{2d}$ correspond to tight frames for $\Hd$.
\item A tight frame $V=[v_1,\ldots,v_n]$ for $\Hd$ corresponds to a tight frame
for $\CC^{2d}$ if and only if it satisfies
\begin{equation}
\label{HtoCcdnII}
\sum_j\sum_k |\Co_1(\inpro{v_j,v_k})|^2
=\sum_j\sum_k |\Co_2(\inpro{v_j,v_k})|^2.
\end{equation}
%which is equivalent to
%\begin{equation}
%\label{CtoHframeopcdn}
%V_1V_1^*=V_2V_2^*={1\over2}AI, \qquad V_1V_2^T=V_2V_1^T=0, \qquad
%dA:=\sum_j\norm{v_j}^2.
%\end{equation}
%where $V=V_1+V_2j$, $V_1,V_2\in\CC^{d\times n}$.
%This condition depends only on $V$ up to unitary
%equivalence, and is equivalent to following condition on Gramian
\end{enumerate}
\end{theorem}

\begin{proof} The sequence $V=V_1+V_2j$ is a tight frame for $\Hd$ if and only if
$$ VV^* = (V_1+V_2j) (V_1^*-V_2^Tj) 
= (V_1V_1^* +V_2V_2^*)+(V_2V_1^T -V_1V_2^T)j = AI,  $$
%\begin{align*}
%VV^* &= (V_1+V_2j) (V_1^*-V_2^Tj) 
%= V_1V_1^*-V_1V_2^Tj+V_2jV_1^*-V_2jV_2^Tj\cr
%&= V_1V_1^*-V_1V_2^Tj+V_2V_1^Tj+V_2V_2^*
%= (V_1V_1^* +V_2V_2^*)+(V_2V_1^T -V_1V_2^T)j
%= AI, 
%\end{align*}
which is clearly satisfied if $V$ corresponds to a tight frame for $\CC^{2d}$
(by Lemma \ref{tightframesCtoHlemma}).

The variational characterisation for being a tight frame for $\Hd$ and for $\CC^{2d}$
are %(respectively)
$$ \norm{V^*V}_F^2 = {1\over d}\Bigl(\sum_j\norm{v_j}^2\Bigr)^2, \qquad
\norm{\Co_1(V^*V)}_F^2 = {1\over 2d}
%(\trace(V_\CC^*V_\CC))^2.
\Bigl(\sum_j\norm{v_j}^2\Bigr)^2.  $$
Hence, if $V$ gives a tight frame for $\Hd$, then it gives a tight frame
for $\CC^{2d}$ if and only if
$$ 2\norm{\Co_1(V^*V)}_F^2- \norm{V^*V}_F^2 =0. $$
Since $|\inpro{v_j,v_k}|^2=|\Co_1(\inpro{v_j,v_k})|^2 +|\Co_2(\inpro{v_j,v_k})|^2$,
this is (\ref{HtoCcdnII}).
\end{proof}

The conditions (\ref{CtoRcdnIIextra}) and (\ref{HtoCcdnII})
can be written insightfully as
$$ \norm{\Re(V^*V)}_F=\norm{\Im(V^*V)}_F, \qquad
\norm{\Co_1(V^*V)}_F=\norm{\Co_2(V^*V)}_F. $$

\begin{example} 
Let $V=[1,i,j,k]$, which is a tight frame for $\HH$.
The Gramian is
$$ V^*V
=\pmat{1&i&j&k \cr -i&1&-k&j \cr -j&k&1&-i \cr -k&-j&i&1 }
=\pmat{1&i&0&0 \cr -i&1&0&0 \cr 0&0&1&-i \cr 0&0&i&1 } 
+\pmat{0&0&1&i \cr 0&0&-i&1 \cr -1&i&0&0 \cr -i&-1&0&0 }j, $$
so this gives a tight frame for $\CC^2$, i.e., $W=[e_1,ie_1,e_2,ie_2]$, with Gramian
$$ W^*W=\pmat{ 1&i&0&0 \cr -i&1&0&0 \cr 0&0&1&i \cr 0&0&-i&1 }
=\pmat{ 1&0&0&0 \cr 0&1&0&0 \cr 0&0&1&0 \cr 0&0&0&1 }
+i\pmat{ 0&1&0&0 \cr -1&0&0&0 \cr 0&0&0&1 \cr 0&0&-1&0 }, $$
so that this in turn gives a tight frame for $\RR^4$, i.e., $[e_1,e_3,e_2,e_4]$.
\end{example}

\begin{example} Consider the Gramian of the SIC of four vectors in $\CC^2$
(Example \ref{HoggarlinesinC^2}).
The contribution to $\norm{V^*V}_F$ %the square of Frobenius norm 
of the diagonal entries, 
which are all real, is $4$,
and for the off diagonal entries it is $12{1\over3}=4$.
%The Frobenius norm of the diagonal entries, which are all real, is $4$,
%and for the off diagonal entries it is $12{1\over3}=4$. 
Thus the SIC corresponds
to a tight frame for $\RR^4$ if and only if its vectors can be scaled so
that the off diagonal entries of the Gramian are pure imaginary.
This can in fact be done, e.g., take $V=[v,iSv,i\gO v,-S\gO v]$, to obtain 
$$ \pmat{\Re(V)\cr\Im(V)}
=\pmat{a&-b&0&b \cr b&0&b&-a\cr 0&b&a&b\cr b&a&-b&0},
\qquad a={\sqrt{3+\sqrt{3}}\over\sqrt{6}}, \
b={\sqrt{3-\sqrt{3}}\over2\sqrt{3}}. $$
This is an orthonormal basis, by Proposition \ref{tightframeequidefs},
or directly by using (\ref{CtoRinpro}). Hence there is a norm-preserving
(invertible) $\RR$-linear map $\CC^2\to\RR^4$ which maps the SIC to an
orthonormal basis.
\end{example}

We now summarise some basic results about
$[\cdot]_\FF$, $\FF=\RR,\CC$, and the associated linear maps, in a unified form.
We first observe that in the literature, there is some variation in the definitions,
in particular, the ordering of $[v]_\RR$ can be either of
$$ [v]_\RR =\pmat{\Re(v)\cr\Im(v)}, \qquad \pmat{\Re(v_1)\cr\Im(v_1)\cr\vdots\cr
\Re(v_d)\cr\Im(v_d)}, $$
and similarly for $[v]_\CC$. In the latter case (cf \cite{H76}, \cite{R14}
for $[v]_\CC$), the matrix representation
$[A]_\RR$ is then obtained by replacing the entry $a_{jk}$ of the matrix $A$ by the matrix
$$\pmat{\Re(a_{jk})&-\Im(a_{jk})\cr\Im(a_{jk})&\Re(a_{jk})}. $$
Our choice of the former was governed by the simpler formulas (cf \cite{C80}).
Indeed, with $L=A+iB,A+Bj$ (respectively), we have the explicit formulas
\begin{equation}
\label{singlematrep}
[L]_\FF = \pmat{A&-B\cr\overline{B}&\overline{A}}, \qquad
[L^*]_\FF = \pmat{A^*&B^T\cr -B^*&A^T}=[L]_\FF^*, \qquad
\FF=\RR,\CC.
\end{equation}
%We have not yet needed the inverse of a matrix (or linear map) over $\HH$.
%It can be shown that if $AB=I$ (for matrices over $\HH$) then
%$BA=I$, and so a right inverse exists for $A$ if and only if a left inverse exists, 
%and these inverses are equal (and denoted by $A^{-1}$).

\begin{theorem} 
\label{mapspropspreserved}
The $\FF$-linear maps $[\cdot]_\FF$, $\FF=\RR,\CC$ given by
(\ref{CdR2dcorrespondence}) and (\ref{HdC2dcorrespondence})
have the following properties
\begin{enumerate}[\rm(a)]
\item They map $r$-dimensional subspaces to $(2r)$-dimensional subspaces.
\item They preserve the Euclidean norm of a vector.
\item They map orthogonal vectors to orthogonal vectors.
\item They map tight frames 
satisfying (\ref{CtoRcdnIIextra}) and (\ref{HtoCcdnII}), respectively,
to tight frames.
\item They map equi-isoclinic $r$-subspaces to equi-isoclinic $(2r)$-subspaces.
\item They map equichordal $r$-subspaces to equichordal $(2r)$-subspaces.
% (of twice the dimension).
\end{enumerate}
Moreover, the associated $\FF$-linear maps $L\mapsto[L]_\FF$ to matrices over $\FF$ 
%given by (\ref{singlematrep}) 
satisfy
\begin{enumerate}[\rm(i)]
\item $[AB]_\FF=[A]_\FF[B]_\FF$, 
$[\gl A]_\FF=\gl[A]_\FF$, $\gl\in\RR$, 
and $[A^*]_\FF=[A]_\FF^*$.
\item They map rank $r$ linear maps to rank $2r$ linear maps.
\item They map invertible linear maps to invertible linear maps,
with $[A^{-1}]_\FF=[A]_\FF^{-1}$. % (for $A$ invertible).
\item They map self adjoint operators to self adjoint operators.
\item They map unitary operators to unitary operators.
\item They map orthogonal projections to orthogonal projections, 
and in particular the identity to the identity.
\end{enumerate}
\end{theorem}

\begin{proof} 
For the first part, (a) has already been observed,
(b) and (c) follow directly from
(\ref{CtoRinpro}) and (\ref{HtoCinpro}), 
(d) follows from Theorems 
\ref{tightframesRtoC}
and \ref{tightframesCtoH}, 
and (e) and (f) follow from the definitions
(\ref{isoclinicequivcdn}) and
(\ref{equichordalHdefn}), and the facts (i), (ii), (vi).

Now the second part. The first part of (i) follows from the definition,
and the second part was a calculation that we did in each case.
For (ii), we have $\ker([L]_\FF)=[\ker(L)]_\FF$, and the result follows
from (a), with (iii) being a special case. If $A$ is invertible, then
(i) gives $I=[I]_\FF=[AA^{-1}]_\FF=[A]_\FF[A^{-1}]_\FF$, which gives the
formula for the inverse. The properties (iv), (v) and (vi) are straightforward
calculations using (\ref{singlematrep}).
\end{proof}

\begin{example} From the observation 
$$ j(A_1+A_2 j) = (\overline{A_1}+\overline{A_2}j)j, \qquad
A_1,A_2\in\CC^{m\times n}, $$
it follows that the image of the $m\times n$ matrices over $\HH$ is
$$ [\HH^{m\times n}]_\CC = \{A\in\CC^{2m\times 2n}:
J_m A=\overline{A}J_n\},\qquad
J_\ell:=[jI_\ell]_\CC=\pmat{0&-I_\ell\cr I_\ell&0}. $$
\end{example}

\begin{example}
\label{quaternionsgroupframe}
If $G$ is a group of $d\times d$ matrices over $\CC$ or $\HH$,
then it follows from Theorem \ref{mapspropspreserved} that
$[G]_\FF=\{[g]_\FF:g\in G\}$ is an isomorphic group of
$(2d)\times(2d)$ matrices.
As an example, the quaternions $Q_8=\{\pm1,\pm i,\pm j,\pm k\}$
are generated by $i$ and $j$, and so the groups of unitary matrices
$[Q_8]_\CC$ and $[[Q_8]_\CC]_\RR$ are generated by
$$ [i]_\CC= \pmat{i&0\cr0&-i}, \qquad
[j]_\CC=\pmat{0&-1\cr1&0}, $$
$$ [[i]_\CC]_\RR= \pmat{0&0&-1&0 \cr 0&0&0&1\cr 1&0&0&0\cr 0&-1&0&0}, \qquad
 [[j]_\CC]_\RR= \pmat{0&-1&0&0 \cr 1&0&0&0\cr 0&0&0&-1\cr 0&0&1&0}, \qquad $$
respectively. These %irreducible 
representations of $Q_8$ are well known.
\end{example}

\begin{example}
If $V=[v_1,\ldots,v_n]\in\HH^{d\times n}$ gives a tight frame of $n$ vectors for 
$\Hd$, i.e., $VV^*=AI$, then 
$$ [V]_\CC=\bigl[ [v_1]_\CC,\ldots,[v_n]_\CC,[v_1j]_\CC,\ldots,[v_nj]_\CC\bigr]$$
gives a tight frame of $2n$ vectors for $\CC^{2d}$.
\end{example}

The equiangular lines in $\HH^2$ of \cite{B20} were obtained
by considering equi-isoclinic planes in $\CC^4$. We now 
explain the mechanism.

\begin{example}
Associated with a unit vector $v_a\in\Hd$, we have
$$V_a:=[[v_a]_\CC,[v_aj]_\CC]\in\CC^{2d\times2}, $$ 
with orthonormal columns
which span a plane in $\CC^{2d}$. 
The entries of the ``block Gramian'' for $V=[V_1,\ldots,V_n]$ are
$V_a^*V_b$ (with $V_a^*V_a=I$).
These satisfy
\begin{equation}
\label{hardtoproveactually}
 (V_a^*V_b)^*(V_a^*V_b) =
\pmat{|\inpro{v_a,v_b}_\HH|^2 & 0 \cr 0&|\inpro{v_a,v_b}_\HH|^2},
\end{equation}
so that
$$ |\inpro{v_a,v_b}|^2=\gl \Iff (V_a^*V_b)^*(V_a^*V_b) = \gl I. $$
Thus $(v_a)$ gives a set of equiangular lines in $\Hd$ if and only if 
the off diagonal entries of the block Gramian $[V_1,\ldots,V_n]^*[V_1,\ldots,V_n]$
are unitary matrices, up to a fixed scalar.
An $n\times n$ block matrix with this structural form ($2\times2$ blocks,
positive semi-definite of rank $2d$), which corresponds to
equi-isoclinic planes in $\CC^{2d}$,  can then be mapped back (under
$[\cdot]_\CC^{-1}$) to the Gramian of $n$ equiangular lines in $\Hd$,
see Theorem 13, \cite{B20}.
\end{example}

The equation (\ref{hardtoproveactually}) follows by a direct calculation, e.g.,
using (\ref{HtoCinpro}), we have
\begin{align*}
(V_b^*V_aV_a^*V_b)_{11}
&= [v_b]_\CC^*[v_a]_\CC[v_a]_\CC^*[v_b]_\CC +[v_b]_\CC^*[v_aj]_\CC[v_aj]_\CC^*[v_b]_\CC \cr
&= \inpro{ [v_b]_\CC , [v_a]_\CC }_\CC \inpro{ [v_a]_\CC , [v_b]_\CC }_\CC
+\inpro{ [v_b]_\CC , [v_aj]_\CC }_\CC \inpro{ [v_aj]_\CC , [v_b]_\CC }_\CC \cr
&= |\Co_1(\inpro{v_a,v_b}_\HH)|^2
+|\Co_2(\inpro{v_a,v_b}_\HH)|^2=|\inpro{v_a,v_b}_\HH|^2,
\end{align*}
\begin{align*}
(V_b^*V_aV_a^*V_b)_{12}
&= [v_b]_\CC^*[v_a]_\CC[v_a]_\CC^*[v_bj]_\CC +[v_b]_\CC^*[v_aj]_\CC[v_aj]_\CC^*[v_bj]_\CC\cr
&= 
\inpro{ [v_b]_\CC , [v_a]_\CC }_\CC \inpro{ [v_a]_\CC , [v_bj]_\CC }_\CC
+\inpro{ [v_b]_\CC , [v_aj]_\CC }_\CC \inpro{ [v_aj]_\CC , [v_bj]_\CC }_\CC \cr
&= \Co_1(\inpro{v_b,v_a}_\HH) (-\Co_2(\inpro{v_a,v_b}_\HH))
-\Co_2(\inpro{v_b,v_a}_\HH) \Co_1(\inpro{v_aj,v_bj}_\HH) \cr
&= -\Co_1(\inpro{v_b,v_a}_\HH)\Co_2(\inpro{v_a,v_b}_\HH)
	+\Co_2(\inpro{v_b,v_a}_\HH) \Co_1(\inpro{v_b,v_a}_\HH)=0, \cr
\end{align*}
where in the second to last equality we used $\Co_2(\overline{q})=-\Co_2(q)$, $q\in\HH$.

%\begin{align*}
%V_b^* V_a V_a^* V_b
%&= \pmat{ [v_b]^*[v_a]& [v_b]^*[v_aj]\cr [v_bj]^*[v_a]&[v_bj]^*[v_aj]} 
% \pmat{ [v_a]^*[v_b]& [v_a]^*[v_bj]\cr [v_aj]^*[v_b]&[v_aj]^*[v_bj]}  \cr
%& = \pmat{ 
%[v_b]^*[v_a][v_a]^*[v_b] +[v_b]^*[v_aj][v_aj]^*[v_b]
% & [v_b]^*[v_a][v_a]^*[v_bj] +[v_b]^*[v_aj][v_aj]^*[v_bj]
% \cr 3 & 4} 
%\end{align*}

Here is a construction of equiangular lines % obtained by
going in the opposite direction.

\begin{example}
We consider the construction of $64$ equiangular lines in $\CC^8$ by %Hoggar 
\cite{H98}. These were obtained by finding $64$ unit vectors in $\HH^4$ 
with angles ${1\over 9},{1\over3}$ (as vertices of a quaternionic polytope).
These were then mapped by $[\cdot]_\CC$ to $64$ equiangular vectors in $\CC^8$.
We note that for $v,w\in\Hd$, $\ga\in\HH$, (\ref{HtoCinpro}) gives
\begin{align*}\inpro{[v\ga]_\CC,[w]_\CC}_\CC
&= \Co_1(\inpro{v\ga,w}_\HH)= \Co_1(\inpro{v,w}_\HH\ga) \cr
&= \Co_1(\ga) \Co_1(\inpro{v,w}_\HH)
- \Co_2(\ga) \overline{\Co_2(\inpro{v,w}_\HH)},
\end{align*}
so that multiplying vectors in $\Hd$ by noncomplex unit scalars in $\HH$
can change the angle between their images in $\CC^{2d}$. 
\end{example}

\section{Group frames and $G$-matrices}

Many tight frames of interest are the orbit of one or more vectors
under the unitary action of a group, e.g., the Weyl-Heisenberg SICs. 
There is a well developed 
theory of such group frames based in the theory of group representations
(over $\RR$ and $\CC$)
\cite{VW05}, \cite{W13}, \cite{VW16}, \cite{W18}. 
We now give an indication of how this theory extends to representations
over $\HH$ (see \cite{SS95}).

A {\bf representation} of a finite abstract group $G$ on
$\Hd$ is a group homomorphism $\rho:G\to\GL(\Hd)$ from $G$ to the 
invertible $d\times d$ matrices over $\HH$, 
with equivalence defined in the usual way. 
We will consider only {\bf unitary representations}, i.e.. those where
the matrices $\rho(g)$ are unitary.
For these, we will write the unitary action
as $gv:=\rho(g)v$, and we note that $g^*v=g^{-1}v$.
A frame (sequence of vectors) of the form $(gv)_{g\in G}$ 
is said to be a {\bf group frame} (or {\bf $G$-frame}) \cite{W20a}. 
The frame operator of a group frame $(gv)_{g\in G}$ commutes with the 
frame operator, i.e.,
\begin{equation}
\label{SandGcommute}
S(hv) =\sum_{g\in G} gv\inpro{ gv , hv }
=h\sum_{g\in G} h^{-1}gv\inpro{ h^{-1}gv , v }
=h S(v), \quad h\in G, \ v\in\Hd.
\end{equation}
The Gramian of a group matrix has entries of the form
$$ \inpro{hv,gv}=\inpro{g^{-1}hv,v}. $$
A matrix $A=[a_{gh}]_{g,h\in G}\in \HH^{G\times G}$ is a 
{\bf $G$-matrix} (or {\bf group matrix}) if there exists a function
$\nu:G\to\HH$ such that
$$ a_{gh}=\nu(g^{-1}h), \qquad\forall g,h\in G. $$
The Gramian of a $G$-frame is a $G$-matrix, and conversely
if the Gramian of a frame $(v_g)_{g\in G}$ with vectors indexed by $G$
is a $G$-matrix, then it is a $G$-frame (adapt the proof of \cite{W18} Theorem 10.3).
An action (representation) of $G$ on $\Hd$ is {\bf irreducible} if the only $G$-invariant subspaces 
of $\Hd$ are $0$ and $\Hd$, i.e., $\spam_\HH\{gv\}_{g\in G}=\Hd$, for all $v\ne0$.

The theory of $G$-frames
for real and complex actions 
begins with irreducible actions, where it takes its simplest form.
This extends without issue.

\begin{proposition}
Suppose that a unitary action of a group $G$ on $\Hd$ is
irreducible. Then $(gv)_{g\in G}$ is a tight $G$-frame for $\Hd$ for 
any $v\ne 0$, i.e.,
$$ x={d\over |G|} {1\over\norm{v}^2}\sum_{g\in G} gv\inpro{gv,x}, \qquad
\forall x\in\Hd. $$
\end{proposition}

\begin{proof} Fix $v\ne0$, and let $S$ be the frame operator of $(gv)_{g\in G}$.
Since $S$ is nonzero and positive semidefinite, it has an eigenvalue $\gl>0$, 
with corresponding eigenvector $w$. By (\ref{SandGcommute}), 
$S$ commutes with the action of $g\in G$, so that %and so we have
$$ S(gw)=g(Sw)=g(w\gl)=(gw)\gl, $$
so that $gw$ is an eigenvector for $\gl$. But $(gw)_{g\in G}$ spans $\Hd$, 
so that $S=\gl I$, i.e., $(gv)_{g\in G}$ is a tight frame.
Since $S$ is Hermitian, taking the trace gives
$$ \trace(S)=\Re(\trace(S)) =\sum_g \norm{gv}^2=|G|\, \norm{v}^2=\trace(\gl I)= d\gl, $$
which gives the value of $\gl$.
\end{proof}

The general theory \cite{VW16}, \cite{W18}, which allows for multiple orbits, 
involves the decomposition of the vector space into irreducible $G$-invariant subspaces.

\begin{example}
Each finite subgroup of $\HH^*$ corresponds to a (faithful) irreducible action
on $\HH^1$. These subgroups were classified by Stringham \cite{S81}. 
They are the infinite families of cyclic groups (generated by the $n$-th roots of unity)
and binary dihedral groups, together with the binary tetrahedral, octahedral and icosahedral groups.
%On $\CC^2$ these correspond to the known irreducible representations.
\end{example}

\begin{example}
The group generated by the matrices
$$ \pmat{0&1\cr1&0}, \quad
\pmat{1&0\cr0&i},  \quad
\pmat{1&0\cr0&j}, $$
has an irreducible unitary action on $\HH^2$.
It consists of all $128$ invertible matrices with two zero entries and two entries in $Q_8$.
It contains the scalar matrices from $Q_8$ and its center is $\pm I$.
Thus each orbit can be viewed as $16$ lines in $\HH^2$ (as a left vector space).
This is an example of a (quaternionic) reflection group, i.e., 
a finite group generated by reflections (linear maps which act 
as the identity on a hyperplane). 
The finite irreducible quaternionic reflection groups
have been classified (up to conjugacy) by Cohen \cite{C80}.
\end{example}

It is expected that the highly symmetric tight frames of \cite{BW13} corresponding
to complex reflection groups could be extended to the quaternionic reflection groups.
In this regard we note the regular quaternionic polytopes have been classified by % Cuypers 
\cite{C95}.

For $G$ abelian, there are a finite number of tight $G$-frames (called harmonic frames)
that can be obtained by ``taking rows of the character table''
(see \cite{VW05}, \cite{CW11}). We now give an example to show how this can
be extended to the quaternionic setting.

\begin{example}
\label{Hharmonicframes}
 (Quaternionic harmonic frames).
The irreducible representations over $\CC$ for an abelian group $G$ are all one-dimensional
(this characterises abelian groups), and these ``rows'' of the character table 
are orthogonal, so by taking a set of rows of the character table one obtains a
tight $G$-frame. Consider the quaternion group $G=Q_8$. This has four $1$-dimensional
and one $2$-dimensional irreducible representations over $\CC$.
The $2$-dimensional absolutely irreducible representation splits
into four $1$-dimensional representations over $\HH$, corresponding to the 
outer automorphisms of the quaternions. In this way, one obtains a character table
$$
\begin{array}{lcccccccc}
q\in Q_8 & 1 & -1 & i & -i & j & -j & k & -k \\
\hline
\chi_1 & 1 & 1 & 1 & 1 & 1 & 1 & 1 & 1 \\
\chi_2 & 1 & 1 & 1 & 1 & -1 & -1 & -1 & -1 \\
\chi_3 & 1 & 1 & -1 & -1 & 1 & 1 & -1 & -1 \\
\chi_4 & 1 & 1 & -1 & -1 & -1 & -1 & 1 & 1 \\
\chi_5 & 1 & -1 & i & -i & j & -j & k & -k \\
\chi_6 & 1 & -1 & j & -j & i & -i & -k & k \\
\chi_7 & 1 & -1 & -i & i & k & -k & j & -j \\
\chi_8 & 1 & -1 & k & -k & -i & i & -j & j 
\end{array}
$$
where the rows are orthogonal (cf \cite{SS95}). Taking rows gives a $G$-frame. 
The columns of the character table are also orthogonal, so taking columns also 
gives a tight frame, but these are not $G$-frames, in general (as follows 
for abelian groups by Pontryagin duality).
As an example, the frame obtained by taking the characters $\chi_1$ and $\chi_5$ (rows $1$ and
$5$ of the character table) gives a (unit-norm) tight $Q_8$-frame for $\HH^2$, 
with the inner products $\{1\pm 1,1\pm i, 1\pm j,1\pm k\}$
occurring exactly once in every row (column) of the Gramian. 
This frame has two angles: each vector is orthogonal to one other, and makes a fixed angle 
with all the others.
%By comparison, taking columns $1$ and $3$ gives a unit-norm tight frame for $\HH^2$, 
%which is not a $G$-frame for any $G$, since its Gramian is not a $G$-matrix.
%The product of (linear) characters is no longer a character, though the 
%columns of the character table are closed under pointwise multiplication
%(they form a group isomorphic to $G$).
\end{example}

\section{Projective unitary equivalence}
\label{projectequivsect}

Finally, we consider the equivalence of vectors thought of as lines in $\Hd$. 
Here the noncommutativity of scalar multiplication considerably complicates 
the theory.

We say that two sequences of vectors $(v_j)$ and $(w_j)$ in $\Hd$ are 
{\bf projectively unitarily equivalent} if there exists a unitary $U$ and unit norm scalars
$\ga_j$ with
$$ w_j = (Uv_j)\ga_j, \qquad\forall j. $$
Clearly, projective unitary equivalence is an equivalence relation. 
Moreover, one can define a {\bf projective unitary symmetry group} of $(v_j)_{j\in J}$
to be all the permutations $\gs:J\to J$ for which $(v_j)$ and $(v_{\gs j})$ are 
projectively unitarily equivalent (cf \cite{CW18}).

To make a workable theory, one now needs a way to recognise projective unitary equivalence.
In terms of the Gramians $V=[v_j]$ and $W=[w_j]$, the formal definition says that
$$ W^*W = C^*V^*U^* UVC= C^* V^*V C, \qquad C:=\diag(\ga_j), $$
i.e.,
\begin{equation}
\label{projunitequivGramian}
\inpro{w_j,w_k}=\overline{\ga_k}\inpro{v_j,v_k}\ga_j.
\end{equation}
This leads to a ``linear system'' $C (W^*W) = (V^*V) C$ in the scalars $\ga_j$. However,
due to the noncommutativity of the quaternions, this can not be solved by Gauss elimination,
unless one 
%first applies the complexification map $[\cdot]_\CC$ obtain a linear system over $\CC$. 
first converts it to a linear system over $\RR$ (in the coordinates of the $\ga_j$).
What is usually done in the real and complex cases is to consider a
collection of invariants: the $m$-products, which completely 
characterise  projective unitary equivalence \cite{CW16}.
We now look at
the analogue of these (also see \cite{KMW19} for subspaces of $\Cd$).

For a sequence of vectors $(v_j)$ in $\Hd$ the {\bf $m$-products} are
$$ \gD(v_{j_1},v_{j_2},\ldots,v_{j_m})
:=
\inpro{v_{j_1},v_{j_2}}
\inpro{v_{j_2},v_{j_3}}
\inpro{v_{j_3},v_{j_4}}
\cdots \inpro{v_{j_m},v_{j_1}}\in\HH. $$
The $1$-products and $2$-products are clearly projective unitary invariants, since
$$ \gD(v_j)=\norm{v_j}^2, \qquad \gD(v_j,v_k)=|\inpro{v_j,v_k}|^2. $$
From these, we can define the {\bf frame graph} of $(v_j)$ to be the graph with
vertices $\{v_j\}$ and
an edge between $v_j$ and $v_k$ ($j\ne k$) if and only if 
$\inpro{v_j,v_k}\ne 0$.

Further, since
\begin{align*}
\gD\bigl((U &v_{j_1})\ga_{j_1},(Uv_{j_2})\ga_{j_2},\ldots,(Uv_{j_m})\ga_{j_m}\bigr) \cr
&= 
\inpro{(Uv_{j_1})\ga_{j_1},(Uv_{j_2})\ga_{j_2}} 
\inpro{(Uv_{j_2})\ga_{j_2},(Uv_{j_3})\ga_{j_3}} \cdots 
\inpro{(Uv_{j_m})\ga_{j_m},(Uv_{j_1})\ga_{j_1}} \cr
&= 
\overline{\ga_{j_1}} \inpro{v_{j_1} ,v_{j_2} } \ga_{j_2}
\overline{\ga_{j_2}} \inpro{v_{j_2} ,v_{j_3} } \ga_{j_3}
\cdots 
\overline{\ga_{j_m}} \inpro{v_{j_m} ,v_{j_1} } \ga_{j_1} \cr
&=\overline{\ga_{j_1}} \gD(v_{j_1},v_{j_2},\ldots,v_{j_m}) \ga_{j_1} \cr
&=\ga_{j_1}^{-1} \gD(v_{j_1},v_{j_2},\ldots,v_{j_m}) \ga_{j_1},
\end{align*}
the $m$-products are projective unitary invariants of $(v_j)$ up to similarity
(congruence),
and real frames are characterised by having real $m$-products.
Since a quaternion $q$ is determined up to similarity 
(conjugation in $\HH\setminus\{0\}$, which is equivalent 
to conjugation by unit scalars)  by its real part $\Re(q)$
and its norm $|q|$, %and so 
we can define (reduced) $m$-products as a pair of real numbers
$$ \gD_r(v_{j_1},v_{j_2},\ldots,v_{j_m}):=(\Re(q),|q|), \qquad
q= \gD(v_{j_1},v_{j_2},\ldots,v_{j_m}). $$
These are projective unitary invariants. For the complex case, the $m$-products
are projective unitary invariants, which depend only on the cycle $(j_1,\ldots,j_m)$,
and a small set of $m$-products corresponding to a basis for the cycle space of the
frame graph of $(v_j)$ provide a set of invariants which characterise projective
unitary equivalence (see \cite{CW16}). We can not yet make a similar claim in the
quaternionic case, though we do imagine that the $m$-products do characterise projective
unitary equivalence. 

The dependence of $m$-products on only the associated $m$-cycle
in the frame graph does follow, by the calculation (for nonzero $m$-products)
$$ a^{-1} \gD(v_{j_1},v_{j_2},\ldots,v_{j_m}) a
= \gD(v_{j_2},v_{j_3},\ldots,v_{j_m},v_{j_1}) , \qquad
a={\inpro{v_{j_1},v_{j_2}}\over|\inpro{v_{j_1},v_{j_2}}|}, $$
and so, in addition to the $1$-products and $2$-products,  
we need only consider the $m$-products for $m\ge3$ which correspond to 
$m$-cycles in the frame graph, i.e., are nonzero.
To check that the $m$-products for two sequences are equal (up to conjugation),
it suffices to consider only the $m$-products corresponding to a cycle basis 
for the cycle space of the (common) frame graph:

\begin{lemma} (Cycle decomposition) For $1\le k\le m$, $n\ge1$, we have
\begin{align*}
&\gD(v_k,v_{k+1},\ldots,v_m,v_1,\ldots,v_{k-1})\gD(v_k,\ldots,v_1,w_1,\ldots,w_n) \cr
&\qquad = |\inpro{v_1,v_2}|^2|\inpro{v_2,v_3}|^2\cdots|\inpro{v_{k-1},v_k}|^2
\gD(v_k,v_{k+1},\ldots,v_m,v_1,w_1,w_2,\ldots,w_n). 
\end{align*}
\end{lemma}

\begin{proof}
Expanding the left hand side gives
\begin{align*}
&\inpro{ v_k , v_{k+1} } \inpro{ v_{k+1} , v_{k+2} } \cdots 
	\inpro{ v_{m-1} , v_m }
\inpro{ v_m , v_1 }
	\inpro{ v_1 , v_2 }
	\cdots
	\inpro{ v_{k-2} , v_{k-1} }
	\inpro{ v_{k-1} , v_k } \cr
&\qquad\times 
	\inpro{ v_k , v_{k-1} }
	\inpro{ v_{k-1} , v_{k-2} } 
	\cdots 
	\inpro{ v_{2} , v_1 }
\inpro{ v_1 , w_1 }
	\inpro{ w_1 , w_2 }
	\cdots\inpro{ w_{n-1} , w_{n} }
	\inpro{ w_{n} , v_k }, 
\end{align*}
which simplifies to the right hand side, since 
$\inpro{v_{j-1},v_{j}}\inpro{v_{j},v_{j-1}}=|\inpro{v_{j-1},v_{j}}|^2\in\RR$ 
commutes with any quaternion.
\end{proof}

This gives the following condition for projective unitary equivalence.

\begin{theorem}
A necessary condition for sequences $(v_j)$ and $(w_j)$ of $n$ vectors in $\Hd$ to be
projectively unitarily equivalent is that the $m$-products 
corresponding to a cycle basis for the frame graph are
are equal (up to conjugation).
\end{theorem}

In the complex setting, this says that the $m$-products are equal, and the
converse is proved by explicitly constructing scalars $\ga_j$
which satisfy (\ref{projunitequivGramian}). 
The difficulties in extending
this converse to the quaternionic setting include the fact that
for $w_j=(Uv_j)\ga_j$, 
\begin{equation}
\label{gaconstraint}
\gD(w_{j_1},w_{j_2},\ldots,w_{j_m}) 
= \overline{\ga_{j_1}} \gD(v_{j_1},v_{j_2},\ldots,v_{j_m}) \ga_{j_1},
\end{equation}
which puts further constraints on the $\ga_j$
(for $m\ge3$ and the $m$-product nonzero).
Indeed, in the complex setting one can assume that any $\ga_j$ is $1$, 
simply by replacing $U$ by the unitary matrix $\ga_j U$. 
Nevertheless, those parts of the theory that we do have allow us to 
investigate such things as the symmetries of lines, as our final example shows.

\begin{example} 
\label{sixlinesexample}
Consider the six tight equiangular lines in $\HH^2$
at angle $\gl=c^2={2\over5}$ of 
\cite{B20}
$$ 
v_1=\pmat{ 1\cr 0 }, \
v_2=\pmat{ {\sqrt{2}\over\sqrt{5}} \cr {\sqrt{3}\over\sqrt{5}} }, \
v_3=\pmat{ {\sqrt{2}\over\sqrt{5}} \cr -{\sqrt{3}\over4\sqrt{5}}+{3\over4}i }, \
v_4=\pmat{ {\sqrt{2}\over\sqrt{5}} \cr -{\sqrt{3}\over4\sqrt{5}}-{1\over4}i+{1\over\sqrt{2}}j },  $$
$$
v_5=\pmat{ {\sqrt{2}\over\sqrt{5}} \cr -{\sqrt{3}\over4\sqrt{5}}-{1\over4}i-{1\over2\sqrt{2}}j+{\sqrt{3}\over2\sqrt{2}}k }, \
v_6=\pmat{ {\sqrt{2}\over\sqrt{5}} \cr
-{\sqrt{3}\over4\sqrt{5}}-{1\over4}i-{1\over2\sqrt{2}}j-{\sqrt{3}\over2\sqrt{2}}k }, $$
which are said to have ``symmetry group'' $A_6$. 
The reduced $m$-products $\gD_r(v_{j_1},\ldots,v_{j_m})$ of distinct vectors 
for $m=1,2,3,4,6$ are all equal, 
taking the values
$$ (1,1), \quad ({2\over5},c^2), \quad ({1\over10},c^3), \quad (-{1\over50},c^4), 
\quad (-{11\over 250},c^6) $$
respectively, which puts no restriction on the possible projective symmetry group
of the lines. However, the reduced $5$-products (of distinct vectors) take two values
$$  (-{25 \pm 9\sqrt{5}\over 500},c^5), $$
and the permutations of the vectors which preserve these
$5$-products is indeed $A_6$. 
Thus the projective symmetry group is a subgroup of $A_6$.
With the present theory, this does not yet establish that $A_6$ is the 
projective symmetry group.

We now seek a corresponding projective unitary symmetry for each $\gs\in A_6$, 
i.e., a unitary 
matrix $U_\gs$ and corresponding scalars $\ga_j$ (also depending on 
$\gs$) for which 
$$ w_j=v_{\gs j}=(U_\gs v_j)\ga_j, \qquad\forall j. $$
Once the unit scalars $\ga_j$ corresponding to a basis $[v_j]_{j\in J}$ of vectors 
from $(v_j)$ are known, the matrix $U_\gs$ is uniquely determined by
$$ U_\gs [v_j\ga_j]_{j\in J}=[v_{\gs j}]_{j\in J}
\Implies U_\gs =[v_{\gs j}]_{j\in J} [v_j\ga_j]_{j\in J}^{-1}, $$
and it can then be checked whether or not the $U_\gs$ is unitary and permutes the other lines.
By (\ref{gaconstraint}),
for $j,k,\ell$ distinct, the unit scalar $\ga_j$ satisfies 
$$ \ga_j\gD(v_{\gs j},v_{\gs k},v_{\gs\ell})=\gD(v_j,v_k,v_\ell)\ga_j, $$
which gives a homogeneous linear system of four equations for the 
four real coordinates of $\ga_j$. In the cases considered, this had 
a unique solution of unit norm up to a choice of sign, which was 
made in order to obtain a unitary matrix $U_\gs$.
For the generators 
$$ a=(1 2)(3 4) \quad\hbox{(order $2$)}, \qquad
b=(1 2 3 5)(4 6) \quad\hbox{(order $4$)} $$
%$$\hbox{$a=(1 2)(3 4)$ (order $2$) and $b=(1 2 3 5)(4 6)$ (order $4$)} $$
for $A_6$, we obtained
$$ \ga_1=\ga_2=-{\sqrt{2}\over\sqrt{3}}i+{1\over\sqrt{3}}k, 
\qquad U_a = \pmat{
{2\over\sqrt{15}}i-{\sqrt{2}\over\sqrt{15}}j & {\sqrt{2}\over\sqrt{5}}i -{1\over\sqrt{5}}j \cr
{\sqrt{2}\over\sqrt{5}}i -{1\over\sqrt{5}}j
& 
-{2\over\sqrt{15}}i+{\sqrt{2}\over\sqrt{15}}j 
 }, \quad U_a^2=-I, $$ 
and
%for the second ($\ga_2$ needed to be multiplied by $-1$ in the calculation)
$$ \ga_1= {1\over2\sqrt{2}}+{\sqrt{5}\over2\sqrt{6}}i - {3-\sqrt{5}\over4\sqrt{3}}j
-{\sqrt{5}+1\over4}k,\quad
\ga_2={\sqrt{5}\over 4} +{1\over4\sqrt{3}}i-{3\sqrt{5}+1\over4\sqrt{6}}j 
-{\sqrt{5}-1\over4\sqrt{2}}, $$
$$ U_b = \pmat{
{1\over2\sqrt{5}}+{1\over2\sqrt{3}}i+{3-\sqrt{5}\over2\sqrt{30}}j+{\sqrt{5}+1\over2\sqrt{10}}k 
& {\sqrt{3}\over2\sqrt{10}}-{1\over2\sqrt{2}}i+{3+\sqrt{5}\over4\sqrt{5}}j 
-{\sqrt{3}\over5+\sqrt{5}} k \cr 
{\sqrt{3}\over2\sqrt{10}}+{1\over2\sqrt{2}}i +{3-\sqrt{5}\over4\sqrt{5}}j 
+{\sqrt{3}\over5-\sqrt{5}}k &
-{1\over2\sqrt{5}}+{1\over2\sqrt{3}}i-{3\sqrt{5}+5\over10\sqrt{6}}j
+{\sqrt{5}-1\over2\sqrt{10}}k }, \quad U_b^4=-I. $$
These unitary matrices $U_a$ and $U_b$ do give the projective unitary symmetries
supposed. Moreover, they generate the double cover $2\cdot A_6$ of $A_6$,
and so we have verified that $A_6$ is indeed the projective symmetry group
of the six equiangular lines in $\HH^2$. We note that our method did not
require prior knowledge of what the symmetry group was.
\end{example}

The action group of the faithful representation of $2\cdot A_6$
obtained in Example \ref{sixlinesexample} contains $40$ reflections 
(of order $3$), and it is an irreducible reflection group which 
appears on the list of \cite{C80}. 
%Moreover, the vectors giving the lines
The vectors giving the lines
are eigenvectors of nontrivial elements of the group, and so the 
six equiangular lines in $\HH^2$ can be constructed directly 
from the reflection group as a group frame (or even from the abstract group $2\cdot A_6$) \cite{W24}.

%The $1$-products and $2$-products are
%$$ \gD(v_j)=1, \qquad \gD(v_j,v_k)={2\over5}, \quad j\ne k, $$
%and the reduced $3$-products and $4$-products are
%$$ \gD(v_{j_1},v_{j_2},v_{j_3})= ({1\over10},\gl^3), \qquad  $$
%For the six equiangular lines, the $m$-products for $m=1,2,3,4$ are
%$$\{1,{2\over5},{1\over10},-{1\over50}\}$$
%respectively, then for $m=5$
%$$ - {25 \pm 9\sqrt{5}\over 500} \approx -0.00975077641, -0.09024922359. $$
%<S-F6>Finally $m=6$ is $$ -{11\over 250}. $$
% For the six equiangular lines, the $m$-products for $m=1,2,3,4$ are

The sets of five and six equiangular lines in $\HH^2$ were first calculated
in \cite{K08} using the Hopf map. Though this technique does not
immediately generalise to other dimensions, like that of \cite{B20},
we recount the essential details, as it sheds further light on the geometry
of these lines.
The {\bf Hopf map} $\psi$ maps a point $\vec{a}=(a_1,\ldots,a_5)$
on the unit sphere in $\RR^5$ to a line in the projective space $\HH\PP^1$,
i.e., a the unit vector $v\in\HH^2$ in the line with $v_2\ge0$, and is
given by $\psi(0,0,0,0,1):=(1,0)$ and
$$ \psi(\vec{a}) := \pmat{ {a\over\sqrt{2(1-a_5)}}\cr{\sqrt{1-a_5}\over\sqrt{2}}}, 
\qquad a:=a_1+a_2i+a_3j+a_4k, \quad a_5\ne1.  $$
A calculation shows that
\begin{align}
\label{Hopfinpro}
|\inpro{\psi(\vec{a}),\psi(\vec{b})}_\HH|^2
%&=\psi(\vec{a})^*\psi(\vec{b})\psi(\vec{b})^*\psi(\vec{a}) \cr
%&= \Bigl({a^*b\over2\sqrt{1-a_5}\sqrt{1-b_5}}
%+{\sqrt{1-a_5}\sqrt{1-b_5}\over 2} \Bigr) \cr
%& = {a^*bb^*a\over 4(1-a_5)(1-b_5)}+{a^*b+b^*a\over4}  +{(1-a_5)(1-b_5)\over4} \cr
%&= { (1-a_5^2)(1-b_5^2)\over 4(1-a_5)(1-b_5)}+
%{\inpro{[a],[b]}_\RR\over 2} + {(1-a_5)(1-b_5)\over4} \cr
&= {1+\inpro{\vec{a},\vec{b}}_\RR\over 2}, \qquad\forall \vec{a},\vec{b},
\end{align}
so the $n\ge 3$ unit vectors $(v_j)$, $v_j=\psi(\vec{v_j})\in\HH^2$, 
give tight equiangular lines if and only if 
$$ |\inpro{v_j,v_k}|^2 ={1+\inpro{\vec{v_j},\vec{v_k}}\over 2} 
={n-2\over2(n-1)}
\Iff \inpro{\vec{v_j},\vec{v_k}}=-{1\over n-1}. $$
This latter condition says that the vectors $(\vec{v_j})$ are the vertices of a
regular $n$-vertex simplex embedded in the unit sphere in $\RR^5$, 
which can be done for $n=3,4,5,6$, with the corresponding image $(v_j)$ giving 
$n$ tight equiangular lines in $\HH^2$. 
Moreover, for $n=3$ we get real lines
by choosing the simplex in $\{x:x=(x_1,0,0,0,x_5)\}$,
and complex lines for $n=4$
by choosing the simplex in $\{x:x=(x_1,x_2,0,0,x_5)\}$.

\subsection{Concluding remarks}

We have shown how much of the theory of tight frames extends to quaternionic 
Hilbert space, with the characterisation of projective unitary equivalence
of frames being the aspect that most depends intrinsically on the commutativity 
of the complex numbers. The notions of canonical coordinates and the canonical
Gramian \cite{W18} also extend to $\HH$-vector spaces. In particular,
there is a unique $\HH$-inner product for which a finite spanning set for 
an $\HH$-vector space becomes a normalised tight frame.

Our focus has been on group frames and equiangular lines. The maximal set
of six equiangular lines in $\HH^2$ comes as the orbit of a quaternionic
reflection group, just as the SIC of four equiangular lines in $\CC^2$ is
the orbit of a complex reflection group. However, 
the known SICs in $\Cd$ (with one exception) are orbits of the Weyl-Heisenberg group,
which is not a reflection group for $d\ge3$. 
The key to constructing quaternionic
equiangular lines in this way will be knowing ``the right group''. 
This group might come from numerical
constructions, using the techniques of this last section, 
or from the theory of group representations over $\HH$ (which is in
its infancy). The construction of sets of tight quaternionic lines may
also offer insight into Zauner's conjecture.
Another direction of similar interest is that of optimal packings
in quaternionic projective space $\HH\PP^k$.

Many of our results say, in some sense, that ``there is more room in $\Hd$ than in $\Cd$''.
In particular, we offer the following variation of Conjecture \ref{HlinesConjectI}, 
which does not implicitly reference Zauner's conjecture.

\begin{conjecture}
\label{HlinesConjectII}
The maximal number of quaternionic equiangular lines in $\Hd$ is strictly larger
than the maximal number of complex equiangular lines in $\Cd$, for each $d\ge 2$.
\end{conjecture}

%\subsection{Acknowledgements}
%
%We wish to thank the referees of an earlier version of this paper who alerted us to
%the work of \cite{CKM16}, which we were unaware of.

\medskip
\center{{\bfbigtype Conflict of interest statement}}
\smallskip

The author declared that he has no conflict of interest.

\bibliographystyle{alpha}
\bibliography{references}

\begin{thebibliography}{ACFW18}

\bibitem[ACFW18]{ACFW18}
Marcus Appleby, Tuan-Yow Chien, Steven Flammia, and Shayne Waldron.
\newblock Constructing exact symmetric informationally complete measurements
  from numerical solutions.
\newblock {\em J. Phys. A}, 51(16):165302, 40, 2018.

\bibitem[BF03]{BF03}
John~J. Benedetto and Matthew Fickus.
\newblock Finite normalized tight frames.
\newblock {\em Adv. Comput. Math.}, 18(2-4):357--385, 2003.

\bibitem[BW13]{BW13}
Helen Broome and Shayne Waldron.
\newblock On the construction of highly symmetric tight frames and complex
  polytopes.
\newblock {\em Linear Algebra Appl.}, 439(12):4135--4151, 2013.

\bibitem[BW25]{BW25}
Zachary Buckley and Shayne Waldron.
\newblock Quaternionic {MUBs} in {$\HH^2$} and their reflection symmetries.
\newblock preprint, 7 2025.

\bibitem[CKM16]{CKM16}
Henry Cohn, Abhinav Kumar, and Gregory Minton.
\newblock Optimal simplices and codes in projective spaces.
\newblock {\em Geom. Topol.}, 20(3):1289--1357, 2016.

\bibitem[Coh80]{C80}
Arjeh~M. Cohen.
\newblock Finite quaternionic reflection groups.
\newblock {\em J. Algebra}, 64(2):293--324, 1980.

\bibitem[Cuy95]{C95}
Hans Cuypers.
\newblock Regular quaternionic polytopes.
\newblock {\em Linear Algebra Appl.}, 226/228:311--329, 1995.

\bibitem[CW11]{CW11}
Tuan-Yow Chien and Shayne Waldron.
\newblock A classification of the harmonic frames up to unitary equivalence.
\newblock {\em Appl. Comput. Harmon. Anal.}, 30(3):307--318, 2011.

\bibitem[CW16]{CW16}
Tuan-Yow Chien and Shayne Waldron.
\newblock A characterization of projective unitary equivalence of finite frames
  and applications.
\newblock {\em SIAM J. Discrete Math.}, 30(2):976--994, 2016.

\bibitem[CW18]{CW18}
Tuan-Yow Chien and Shayne Waldron.
\newblock The projective symmetry group of a finite frame.
\newblock {\em New Zealand J. Math.}, 48:55--81, 2018.

\bibitem[ET20]{B20}
Boumediene Et-Taoui.
\newblock Quaternionic equiangular lines.
\newblock {\em Adv. Geom.}, 20(2):273--284, 2020.

\bibitem[FIJM24]{FIJM24}
Matthew Fickus, Joseph~W. Iverson, John Jasper, and Dustin~G. Mixon.
\newblock Equi-isoclinic subspaces from symmetry, 2024.

\bibitem[FJMW17]{FJMW17}
Matthew Fickus, John Jasper, Dustin~G. Mixon, and Cody~E. Watson.
\newblock A brief introduction to equi-chordal and equi-isoclinic tight fusion
  frames, 2017.

\bibitem[GMP13]{GMP13}
Riccardo Ghiloni, Valter Moretti, and Alessandro Perotti.
\newblock Continuous slice functional calculus in quaternionic {H}ilbert
  spaces.
\newblock {\em Rev. Math. Phys.}, 25(4):1350006, 83, 2013.

\bibitem[Hog76]{H76b}
S.~G. Hoggar.
\newblock Quaternionic equi isoclinic {$n$}-planes.
\newblock {\em Ars Combin.}, 2:11--13, 1976.

\bibitem[Hog82]{H82}
S.~G. Hoggar.
\newblock {$t$}-designs in projective spaces.
\newblock {\em European J. Combin.}, 3(3):233--254, 1982.

\bibitem[Hog98]{H98}
Stuart~G. Hoggar.
\newblock {$64$} lines from a quaternionic polytope.
\newblock {\em Geom. Dedicata}, 69(3):287--289, 1998.

\bibitem[Hog77]{H76}
S.~G. Hoggar.
\newblock New sets of equi-isoclinic {$n$}-planes from old.
\newblock {\em Proc. Edinburgh Math. Soc. (2)}, 20(4):287--291, 1976/77.

\bibitem[KF08]{K08}
Mahdad Khatirinejad~Fard.
\newblock {\em Regular structures of lines in complex spaces}.
\newblock ProQuest LLC, Ann Arbor, MI, 2008.
\newblock Thesis (Ph.D.)--Simon Fraser University (Canada).

\bibitem[KMW25]{KMW19}
Emily~J. King, Dustin~G. Mixon, and Shayne Waldron.
\newblock Testing isomorphism between tuples of subspaces, 2025.

\bibitem[KPS21]{KPS21}
Hadi Kharaghani, Thomas Pender, and Sho Suda.
\newblock Balancedly splittable orthogonal designs and equiangular tight
  frames.
\newblock {\em Des. Codes Cryptogr.}, 89(9):2033--2050, 2021.

\bibitem[KTS17]{KTS17}
M.~Khokulan, K.~Thirulogasanthar, and S.~Srisatkunarajah.
\newblock Discrete {F}rames on {F}inite {D}imensional {L}eft {Q}uaternion
  {H}ilbert {S}paces.
\newblock {\em Axioms}, 6(4):3, Feb 2017.

\bibitem[LS73]{LS73}
P.~W.~H. Lemmens and J.~J. Seidel.
\newblock Equi-isoclinic subspaces of {E}uclidean spaces.
\newblock {\em Nederl. Akad. Wetensch. Proc. Ser. A {\bf 76} Indag. Math.},
  35:98--107, 1973.

\bibitem[NS22]{NS22}
Tom Needham and Clayton Shonkwiler.
\newblock Admissibility and frame homotopy for quaternionic frames.
\newblock {\em Linear Algebra Appl.}, 645:237--255, 2022.

\bibitem[Rod14]{R14}
Leiba Rodman.
\newblock {\em Topics in quaternion linear algebra}.
\newblock Princeton Series in Applied Mathematics. Princeton University Press,
  Princeton, NJ, 2014.

\bibitem[SS95]{SS95}
G.~Scolarici and L.~Solombrino.
\newblock Notes on quaternionic group representations.
\newblock {\em Internat. J. Theoret. Phys.}, 34(12):2491--2500, 1995.

\bibitem[Str81]{S81}
W.~I. Stringham.
\newblock Determination of the {F}inite {Q}uaternion {G}roups.
\newblock {\em Amer. J. Math.}, 4(1-4):345--357, 1881.

\bibitem[VSSS20]{VSS20}
Virender, S.~K. Sharma, Ghanshyam Singh, and Soniya Sahu.
\newblock On frames in finite dimensional quaternionic {H}ilbert space.
\newblock {\em Palest. J. Math.}, 9(1):511--522, 2020.

\bibitem[VW05]{VW05}
Richard Vale and Shayne Waldron.
\newblock Tight frames and their symmetries.
\newblock {\em Constr. Approx.}, 21(1):83--112, 2005.

\bibitem[VW16]{VW16}
Richard Vale and Shayne Waldron.
\newblock The construction of {$G$}-invariant finite tight frames.
\newblock {\em J. Fourier Anal. Appl.}, 22(5):1097--1120, 2016.

\bibitem[Wal03]{W03}
Shayne Waldron.
\newblock Generalized {W}elch bound equality sequences are tight frames.
\newblock {\em IEEE Trans. Inform. Theory}, 49(9):2307--2309, 2003.

\bibitem[Wal13]{W13}
Shayne Waldron.
\newblock Group frames.
\newblock In {\em Finite frames}, Appl. Numer. Harmon. Anal., pages 171--191.
  Birkh\"auser/Springer, New York, 2013.

\bibitem[Wal18]{W18}
Shayne F.~D. Waldron.
\newblock {\em An introduction to finite tight frames}.
\newblock Applied and Numerical Harmonic Analysis. Birkh\"{a}user/Springer, New
  York, 2018.

\bibitem[Wal20a]{W20a}
Shayne Waldron.
\newblock The {F}ourier transform of a projective group frame.
\newblock {\em Appl. Comput. Harmon. Anal.}, 49(1):74--98, 2020.

\bibitem[Wal20b]{W20b}
Shayne Waldron.
\newblock Spherical designs and their {G}ramian.
\newblock preprint, 1 2020.

\bibitem[Wal24]{W24}
Shayne Waldron.
\newblock The geometry of the six quaternionic equiangular lines in
  $\mathbb{H}^2$, 2024.

\bibitem[Zau10]{Zau10}
Gerhard Zauner.
\newblock {\em Quantum {D}esigns: {F}oundations of a non-commutative {D}esign
  {T}heory}.
\newblock PhD thesis, University of Vienna, University of Vienna, 1 2010.
\newblock English translation of 1999 Doctorial thesis including a new preface.

\bibitem[Zha97]{Z97}
Fuzhen Zhang.
\newblock Quaternions and matrices of quaternions.
\newblock {\em Linear Algebra Appl.}, 251:21--57, 1997.

\end{thebibliography}
\nocite{*}

%\begin{thebibliography}{99}

%\bibitem{LO91}
%H. Liebeck and A. Osborne, 
%The generation of all rational orthogonal matrices,
%{\it Amer.\ Math.\ Monthly} 
%{\bf 98}
%(1991), 
%131--133. 

%\end{thebibliography}

\end{document}